\documentclass[12pt,reqno]{amsart}

\usepackage{color}
\usepackage{amsmath, amsthm, amssymb}
\usepackage{amsfonts}

\usepackage{hyperref}

\usepackage[ansinew]{inputenc}
\usepackage[dvips]{epsfig}
\usepackage{graphicx}
\usepackage[english]{babel}

\theoremstyle{plain}
\newtheorem{thm}{Theorem}[section]
\newtheorem{cor}[thm]{Corollary}
\newtheorem{lem}[thm]{Lemma}
\newtheorem{prop}[thm]{Proposition}

\theoremstyle{definition}
\newtheorem{defi}[thm]{Definition}

\theoremstyle{remark}

\numberwithin{equation}{section}

\newcommand{\average}{{\mathchoice {\kern1ex\vcenter{\hrule height.4pt
width 6pt depth0pt} \kern-9.7pt} {\kern1ex\vcenter{\hrule
height.4pt width 4.3pt depth0pt} \kern-7pt} {} {} }}
\newcommand{\ave}{\average\int}
\def\R{\mathbb{R}}

\begin{document}

\title[The free boundary in the fully nonlinear thin obstacle problem]{The structure of the free boundary in the \\ fully nonlinear thin obstacle problem}

\author[X. Ros-Oton]{Xavier Ros-Oton}
\address{The University of Texas at Austin, Department of Mathematics, 2515 Speedway, Austin, TX 78751, USA}
\email{ros.oton@math.utexas.edu}

\author[J. Serra]{Joaquim Serra}
\address{Weierstra{\ss} Institut f\"ur Angewandte Analysis
und Stochastik, Mohrenstrasse 39, 10117 Berlin, Germany}
\email{joaquim.serra@upc.edu}




\maketitle

\begin{abstract}
We study the regularity of the free boundary in the fully nonlinear thin obstacle problem.
Our main result establishes that the free boundary is $C^1$ near regular points.
\end{abstract}

\vspace{4mm}

\section{Introduction}

The aim of this paper is to study the regularity of free boundaries in thin obstacle problems.

\subsection{Known results}

The first regularity results for thin obstacle problems were already established in the seventies by Lewy \cite{L}, Frehse \cite{F}, Caffarelli \cite{C}, and Kinderlehrer \cite{K}.
In particular, for the Laplacian $\Delta$, it was proved in \cite{C} that solutions are $C^{1,\alpha}$, for some small $\alpha>0$.

The regularity of free boundaries, however, was an open problem during almost 30 years.
One of the main difficulties in the understanding of free boundaries in thin obstacle problems is that there is not an a priori preferred order at which the solution detaches from the obstacle (blow-ups may have different homogeneities), as explained next.

In the classical (thick) obstacle problem it is not difficult to show that
\begin{equation}\label{classical}
0<cr^2\leq \sup_{B_r(x_0)}u\leq Cr^2
\end{equation}
at all free boundary points $x_0$, where $u$ is the solution of the problem (after subtracting the obstacle $\varphi$).
Then, thanks to this, the blow-up sequence $u(x_0+rx)/r^2$ converges to a global solution $u_0$, and such solutions $u_0$ can be shown to be convex and completely classified; see \cite{C-obst2,C-obst3} and \cite{C-obst}.

The situation is quite different in thin obstacle problems, in which one does not have \eqref{classical}.
This was resolved for the first time in Athanasopoulos-Caffarelli-Salsa \cite{ACS}, by using \emph{Almgren's frequency function}.
Thanks to this powerful tool, one may take the blow-up sequence
\[\frac{u(x_0+rx)}{\left(\ave_{\partial B_r(x_0)}u^2\right)^{1/2}},\]
and it converges to a \emph{homogeneous} function $u_0$ of degree $\mu$, for some $\mu>1$.
Then, by analyzing an eigenvalue problem on $S^{n-1}$, one can prove that
\[\mu<2\qquad \Longrightarrow \qquad \mu=\frac32,\]
and for $\mu=\frac32$ one can completely classify blow-ups.
This leads to the optimal $C^{1,\frac12}$ regularity of solutions and, using also a boundary Harnack inequality in ``slit'' domains, to the $C^{1,\alpha}$ regularity of the free boundary near \emph{regular points} ---those at which $\mu<2$.

After the results of \cite{ACS}, further regularity results for the free boundary have been obtained in \cite{CSS}, \cite{GP}, \cite{GPS}, \cite{DS}, \cite{KPS}, \cite{KRS} and \cite{BFR}.

\subsection{Our setting}

In this paper we study the fully nonlinear thin obstacle problem
\begin{equation}\label{obst-pb}\left\{\begin{array}{rclll}
F(D^2u)&\leq&0\,&\textrm{in} &B_1\\
u&\geq& \varphi\,&\textrm{on} &B_1\cap \{x_n=0\},\\
F(D^2u)&=&0 \,&\textrm{in} &B_1\setminus \{(x',0)\,:\, u(x',0)=\varphi(x')\}.
\end{array}\right.\end{equation}
Here, $x=(x',x_n)\in\R^n$.
When $u$ is even with respect to the variable $x_n$, then the problem is equivalent to
\begin{equation}\label{obst-pb2}\begin{array}{rclll}
F(D^2u)&=&0\,&\textrm{in} &B_1\cap \{x_n>0\}\\
\min(-u_{x_n},\,u-\varphi)&=& 0\,&\textrm{on} &B_1\cap \{x_n=0\}.
\end{array}\end{equation}

Problem \eqref{obst-pb2} was studied in \cite{MS}, where Milakis and Silvestre proved that solutions $u$ are $C^{1,\alpha}(B_{1/2})$ (for some small $\alpha>0$) by following the ideas of \cite{C}.
More recently, Fern\'andez-Real extended the results of \cite{MS} to the general nonsymmetric setting~\eqref{obst-pb} in~\cite{Fer}.

Still, nothing was known about the regularity of the free boundary for this problem.
The main difficulty in the study of such nonlinear thin free boundary problems is the lack of monotonicity formulas for fully nonlinear operators, which makes the proofs of \cite{ACS} non-applicable to the nonlinear setting.

\subsection{Main results}

We present here a new approach towards the regularity of thin free boundaries, and prove that for problem \eqref{obst-pb} the free boundary is $C^1$ near regular points.

As in \cite{MS,Fer}, we assume that the fully nonlinear operator $F$ satisfies:
\begin{equation}\label{F}
F\mbox{ is convex, with ellipticity constants }\lambda, \Lambda,\mbox{ and }F(0)=0.
\end{equation}
Our main result reads as follows.

\begin{thm}\label{thm1}
Let $F$ be as in \eqref{F}.
There exists $\epsilon_0=\epsilon_0(\lambda,\Lambda)>0$ for which the following holds.

Let $u\in C(B_1)$ be any solution of \eqref{obst-pb}, with $\varphi\in C^{1,1}$.
Then, at each free boundary point $x_0\in\partial\{u=\varphi\}\cap B_{1/2}\cap \{x_n=0\}$ we have the following dichotomy:
\begin{itemize}
\item[(i)] either \vspace{-3mm}
\[\sup_{B_r(x_0)}(u-\varphi)\geq c\,r^{2-\epsilon_0},\]
with $c>0$,\vspace{2mm}
\item[(ii)] or \hspace{30mm} $\sup_{B_r(x_0)}(u-\varphi)\leq C_\epsilon \,r^{2-\epsilon}$ \quad for all\quad $\epsilon>0$.
\end{itemize}
Moreover, the set of points $x_0$ satisfying (i) is an open subset of the free boundary and it is locally a $C^1$ graph.
\end{thm}

Notice that, for the Laplacian $\Delta$, once we know that the free boundary is $C^1$, then it can be proved that it is $C^\infty$; see \cite{DS,KPS} and also \cite{RS-C1}.

On the other hand, when $F$ is the Laplacian $\Delta$, at all free boundary points satisfying (i) the blow up is homogeneous of degree $3/2$, and thus all solutions are $C^{1,\frac12}$.
We do not expect this same exponent $3/2$ for all nonlinear operators $F(D^2u)$.
A priori, each different operator $F$ could have one (or more) different exponent~$\mu$, and thus in general solutions would be no better than $C^{1,\alpha}$ for some small $\alpha>0$.
Still, we show in Section \ref{sec7} that
\[u\in C^{1,\frac12-\delta}(B_{1/2})\quad \textrm{whenever}\quad |\Lambda-\lambda|\quad\textrm{is small enough;}\]
see Corollary \ref{cor-reg}.

We think that an interesting feature of our proof is that we establish the regularity of the free boundary without proving any homogeneity or uniqueness of blow-ups, a priori they could be non-homogeneous and/or non-unique.
We do not classify blow-ups but only prove that they are 1D on $\{x_n=0\}$, as explained next.

\subsection{The proofs}

To establish Theorem \ref{thm1} we assume that $x_0$ is a regular free boundary point (i.e., (ii) does not hold at $x_0$), and do a blow-up.
We have to do the blow-up along an appropriate subsequence, so that we get in the limit a global \emph{convex} solution to \eqref{obst-pb}, with zero obstacle, and with \emph{subquadratic growth} at infinity.
Then, we need to prove that blow-ups are 1D on $\{x_n=0\}$, that is, the blow-up $u_0$ is a 1D function on $\{x_n=0\}$, and in particular the contact set $\Omega^*=\{u_0=0\}\cap\{x_n=0\}$ is a half-space.

To do this, we first notice that by a blow-down argument we may reduce to the case in which the convex set $\Omega^*$ is a convex cone $\Sigma^*$.
Then, we separate into two cases, depending on the ``size'' of the convex cone $\Sigma^*$.
If $\Sigma^*$ has zero measure, then $u_0$ is in fact a global solution, and has subquadratic growth.
By $C^2$ regularity estimates this is not possible, and thus $\Sigma^*$ can not have zero measure.
If $\Sigma^*$ has nonempty interior, by convexity of $u_0$ this means that we have a cone of directional derivatives satisfying $\partial_e u_0\geq0$ in $\R^n$.
Then, by a boundary Harnack type estimate (that we also establish here), we prove that all such derivatives have to be comparable in $\R^n$, and that this yields that the cone must be a half-space.

Once we have that blow-ups are 1D on $\{x_n=0\}$, we show that the free boundary $\partial\{u=\varphi\}$ is Lipschitz in a neighborhood of any regular point $x_0$, and $C^1$ at that point.
Finally, by a barrier argument we show that the regular set is open ---with all points in a neighborhood satisfying a uniform nondegeneracy condition.
From here, we deduce that the free boundary is $C^1$ at every point in a neighborhood, with a uniform modulus of continuity.

Notice that an important step in the previous argument is the boundary Harnack type result for the derivatives $\partial_eu_0$, which solve an equation with bounded measurable coefficients in non-divergence form.
The boundary Harnack principle for non-divergence equations is known to be false in $C^{0,\alpha}$ domains of $\R^n$ whenever $\alpha\leq\frac12$; see \cite{BB}.
Still, we prove here that a weaker version of the boundary Harnack principle holds in ``slit'' domains of the form $\R^n\setminus\Sigma^*$, where $\Sigma^*\subset\R^{n-1}\times\{0\}$ is a convex cone.
The proof of such boundary Harnack type estimate is new, and we think it could be of independent interest.

Finally, notice also that our boundary Harnack type result allows us to show that blow-ups are 1D, but does \emph{not} yield the $C^{1,\alpha}$ regularity of free boundaries.
This is because the constants in such boundary Harnack estimate degenerate as the cone $\Sigma^*$ contains two rays forming an angle approaching $\pi$.

\subsection{Plan of the paper}

The paper is organized as follows.

In Section \ref{sec2} we construct some barriers that are needed in our proofs, and prove a maximum principle in $\R^n_+$ for functions $u$ with sublinear growth.
In Section \ref{sec3} we establish our boundary Harnack type inequality for non-divergence equations with bounded measurable coefficients.
In Section \ref{sec4} we prove that global convex solutions with subquadratic growth to the fully nonlinear thin obstacle problem are necessarily 1D on $\{x_n=0\}$.
In Section \ref{sec5} we show that at any regular free boundary point there is an appropriate rescaling such that the rescaled solutions converge in the $C^1$ norm to a global convex solution with subquadratic growth.
In Section \ref{sec6} we prove that the free boundary is flat Lipschitz by combining the results of Section~\ref{sec5} with a maximum principle argument.
Finally, in Section \ref{sec7} we show by a barrier argument that the regular set is open, which yields the $C^1$ regularity of the free boundary.

\section{Preliminaries and tools}
\label{sec2}

We prove here some results that will be used in the paper.
We will denote
\[M^+ u = M^+(D^2 u)\quad \mbox{and}\quad  M^- u = M^-(D^2 u),\]
the Pucci extremal operators; see \cite{CC} for their definition and basic properties.

Throughout the paper we call constants depending only on the dimension $n$ and the ellipticity constants $\lambda,\Lambda$ {\em universal constants}. 
Also, we denote $B^+$ the half ball $B\cap \{x_n>0\}$, where $B$ is some ball centered at some point on $\{x_n=0\}$, and we denote by $B^*$, $\Sigma^*$, and $\Omega^*$, ``thin'' balls, cones, and sets contained on $\{x_n=0\}$.

\subsection{Barriers}

We first construct two barriers.

\begin{lem} \label{lemsupersol0}
For $N = (n-1)\Lambda/\lambda $  the function
\[ \phi_0(x) =
\begin{cases}
\min\{1,  |x'|^2 + N (2x_n-x_n^2)\} \quad&\mbox{in }|x'|\le 1, \ 0\le x_n\le1
\\
1  &\mbox{elsewhere in } x_n\ge0
\end{cases}\]
is continuous (viscosity) supersolution of $M^+ \phi_0\le 0$ in $x_n>0$.
\end{lem}

\begin{proof}
We note that $|x'|^2 + N (2x_n-x_n^2) \ge |x'|^2+ |x_n|^2 \ge |x|$ and thus $\phi_0$ is continuous. 
Also, where $\phi_0<1$ we have $M^+\varphi_0 = 2(n-1)\Lambda - 2N\lambda \le0$. 
Thus, using that the minimum of two supersolutions is a supersolution we easily obtain that $M^+ \phi_0\le 0$ in all of $\R^n$.
\end{proof}

\begin{lem} \label{lemsupersol}
Let $a_i\ge0$ with $\sum_{i=0}^\infty a_i  <\infty$.
Then, the function
\[\phi (x) = \sum_{i=0}^k  2^{i}a_i\phi_0(2^{-i}x)\]
is a continuous (viscosity) supersolution of $M^+\phi \le 0$ in all of $x_n>0$. Moreover, $\phi$ satisfyies
\begin{equation}
2^j a_j \le \phi\quad \mbox{in }\overline{B_{2^{j+1}}^+\setminus B_{2^j}^+}
\end{equation}
and
\begin{equation}\label{control}
\phi \le C\left(\sum_{i=0}^{j} 2^{i}a_i + \sum_{i=j}^\infty a_i \right) \quad \mbox{in }\overline{B_{2^j}^+}
\end{equation}
where $C$ is a universal constant.
\end{lem}

\begin{proof}
Let $\phi_0$ be the supersolution from Lemma \ref{lemsupersol0}.
We then consider, for $k\ge0$
\[
\phi^{k}(x) = \sum_{i=0}^k  2^{i}a_i\phi_0(2^{-i}x)
\]

On one hand, we have
\[
M^+ \phi^{k}(x) \le  \sum_{i=0}^k  2^{i-2}a_i M^+\phi_0(2^{-i}x) \le 0.
\]
On the other hand, whenever $k\ge j$ and $|x|\ge 2^j$ we have
\begin{equation}\label{bybelow}
\phi^{k}(x)\ge 2^{j}a_j\phi_0(2^{-j}x) \ge 2^ja_j,
\end{equation}
since we readily check that $\phi_0 \ge \min\{1,|x'|^2+|x_n|^2\}=1$ outside $B_1^+$ (in a$x_n>0$).

Finally, we note that $\phi_0 \le C \min\{1, |x'|+|x_n|\}$ and thus
\begin{equation}\label{byabove}
\phi^{k}(x) \le C\sum_{i=0}^k  2^{i}a_i\min\{1,2^{-i}|x|\} \le C\left(\sum_{i=0}^{j} 2^{i}a_i + \sum_{i=j}^\infty a_i \right) \quad \mbox{for }x\in B_{2^j}^+.
\end{equation}
Then, the monotone increasing sequence $\phi^k$ converges locally uniformly in $\{x_n>0\}$ to some function $\phi=\phi^\infty$. By the stability of viscosity supersolutions under uniform convergence we have  $M^+\phi\le0$ in all of $\R^n$. That $\phi$ satisfies the other conditions of the lemma is easily verified letting $k\to\infty$ in \eqref{bybelow} and \eqref{byabove}.
\end{proof}

The following subsolution will be used in the proof of our boundary Harnack inequality.

\begin{lem} \label{lemsubsol}
Given $\rho\in(0,1)$ and a ball $B^*= B^*_r(z)$ ($z\in\R^{n-1}$), with $B^*$ contained in $B_1^*$,  there is a function $\phi\in C(B_1)$ satisfying
\begin{equation}\label{subsol1}
\begin{cases}
M^- \phi \ge \chi_{B_{1-\rho}} \quad &\mbox{in }B_1\setminus {B^*} \\
\phi \ge 0 & \mbox{in }B_1 \\
\phi \le C \chi_{B^*}  &\mbox{on } B^*_{1}\\
\phi = 0  &\mbox{on }{\partial B_{1}}
\end{cases}
\end{equation}
where $C$ depends only on $\rho$, $B^*$ and universal constants.
\end{lem}

\begin{proof}
Let $g_0$ be the restriction to $\partial B_1^+$ of the function $\max\{0, 1-(x-z)^2/r^2\}$ and $f_0(x) = f_0(|x|)$ be a radial nonincreasing function with $f_0=0$ for $|x|\ge 1-\rho/2$ and $f_0=1$ for $|x|\le 1-\rho$.

For $\kappa\in(0,1)$ small, we let $\psi$ be the solution to
\begin{equation}\label{pbtau}
\begin{cases}
M^-\psi_\kappa = \kappa f_0\quad& \mbox{in }B_1^+\\
\psi=  g_0 & \mbox{on }\partial B_1^+
\end{cases}
\end{equation}
Let us show that $\kappa$ small enough (depending only on  $\rho$ and $B^*$) we have $\psi\ge 0$ in~$B_1^+$.

Indeed, by the strong maximum principle and Hopf's lemma, for $\kappa=0$ we have
\[
\psi_0 \ge \delta_0 >0 \quad \mbox{in }B_{1-\rho/4}\cap\{x_n> \rho/4\}.
\]
Thus, by the uniqueness of solution to \eqref{pbtau} and the stability of viscosity solutions we deduce that
\begin{equation}\label{zonepsiposi}
\psi_ \kappa \ge \delta_0/2 >0 \quad \mbox{in }B_{1-\rho/4}\cap\{|x_n|> \rho/4\}.
\end{equation}
for $ \kappa$ small.

Next,  for $N$ large enough the function $ \eta =  \exp(-N|x|) -\exp(-N\rho/2)$ satisfies
\begin{equation}\label{Hopf}
M^- \eta = \left( \lambda N^2 - \frac{\Lambda N (n-1)}{|x|}\right) \eta >0 \quad\mbox{in } \{|x|\ge \rho/4\}\cap \{\eta>0\}.
\end{equation}
Thus, we have $M^- \eta \ge c>0$ in $\{ \rho/4 \le |x| \le \rho/2\}$ and using $\frac{\delta_0}{2}\eta(x-x_0)$ as a barrier (by below) with $x_0$ on $\{|x'|\le 1-\rho/2, \ x_n = \rho/2\}$, and by \eqref{zonepsiposi} we obtain
\begin{equation}\label{thegoal}
\psi_ \kappa \ge 0 \quad \mbox{in }B_{1-\rho/2}^+
\end{equation}
when $\kappa$ is chosen small enough.

Finally, from \eqref{thegoal} it follows that (still for $\kappa$ small) we have $\psi_\kappa \ge 0$ in all of $B_{1}^+$. Here we are using that $f_0=0$ in the half annulus $B_1\setminus B_{1-\delta/2}$.

To end the proof, we let $\phi$ be the even reflection of the previous $\frac{1}{\kappa}\psi_\kappa$ with respect to the variable $x_n$ multiplied by a large positive constant $C$. Then, using that $\phi$ will have a negative wedge on $B^*_1\setminus B^*$ it not difficult to verify that it will satisfy all the requirements of the lemma.
\end{proof}

\subsection{A maximum principle in $\R^n_+$ and construction of 1D solutions}

We next prove the following.

\begin{lem}\label{lem-sumablegrowth}
Let $u$ satisfy
\begin{equation}\label{sumablegrowth}
\sup_{B_1^+} |u| + \sum_{i=0}^\infty 2^{-i}\sup_{B_{2^{i+1}}^+\setminus B_{2^i}^+} |u| < \infty
\end{equation}
and
\[
\begin{cases}
M^-u \le0  \quad  \mbox{(resp.  $M^+u \ge0$)} \  &\mbox{in }\{x_n>0\}\\
u \ge 0 \quad \mbox{(resp.  $u \le0$)}&\mbox{on }\{x_n=0\}.
\end{cases}
\]
Then, $u\ge 0$ (resp. $u \le0$) in $\{x_n>0\}$.
\end{lem}

For this, we need the following.

\begin{lem}\label{lem-series}
Let $(a_k)$ be a sequence such that $a_k\geq0$ and $\sum_{k\geq1}a_k<\infty$.
Then, there exists a sequence $(b_k)$ such that $b_k/a_k \ge 1$, $\lim b_k/a_k=\infty$, and $\sum_{k\geq1}b_k<\infty$.
\end{lem}

\begin{proof}
The result is probably well known, we give here a proof for completeness.

Let us define $s_k=\sum_{j\geq k}a_j$. Note that may (and do) assume that $s_1= 1$.
Let
\[b_k=\frac{a_k}{\sqrt{\sum_{j\geq k}a_j}}=\frac{a_k}{\sqrt{s_k}}\ge a_k.\]

Notice that $\lim b_k/a_k=\infty$, since $s_k\to 0$.
Then, we have
\begin{equation}\label{telescopic}
b_k=\frac{s_k-s_{k+1}}{\sqrt{s_k}}\leq 2\sqrt{s_k}-2\sqrt{s_{k+1}},
\end{equation}
where we used that $2\sqrt{x}-2\sqrt{y}\geq (x-y)/\sqrt{x}$ for all $x\geq y$ (this follows from the mean value theorem).
Therefore, by \eqref{telescopic}, we find
\[\sum_{k\geq1}b_k\leq 2\sqrt{s_1}<\infty,\]
and the lemma is proved.
\end{proof}

We now give the:

\begin{proof}[Proof of Lemma \ref{lem-sumablegrowth}]
Let $a_i := 2^{-i} \sup_{B_{2^{i+1}}\setminus B_{2^i}} |u|$.
By assumption $\sum a_i<\infty$ and
then, by Lemma \ref{lem-series}, there exists $b_i$ increasing such that $1\le b_i/a_i \to \infty$ and $\sum b_i<\infty$.
Then, we consider
\[ \phi(x) : = - \sup_{B_1^+}|u|-\sum_{i=0}^\infty 2^i b_i \phi_0(2^{-i}x),\]
where $\phi_0$ is the supersolution in the proof of Lemma \ref{lemsupersol}.
Exactly as in the proof of Lemma \ref{lemsupersol} we find that $\phi$ is subsolution in all of $\{x_n>0\}$. Then, using that $u\ge 0$ on $\{x_n=0\}$, that $b_i/a_i \to \infty$, and the maximum principle, we obtain $u\ge -\epsilon\phi$ in all of $\{x_n\ge0\}$ for every $\epsilon>0$. Thus $u\ge 0$ in all of  $\{x_n\ge0\}$.
\end{proof}

As a consequence of Lemma \ref{lem-sumablegrowth}, we find the following.

\begin{prop}[Extensions]\label{prop-extensions}
Given $g: \R^{n-1}\rightarrow \R$ continuous satisfying
\[ \sup_{B_{1}^*} |g| +\sum_{i=0}^\infty 2^{-i}\sup_{B_{2^{i+1}}^*\setminus B_{2^{i}}^*} |g| < \infty   \]
there exist a unique function $u$ belonging to $C(\{x_n>0\})$ which satisfies \eqref{sumablegrowth}
and
\[\begin{cases}
M^+ u = 0 \quad &\mbox{in }\{x_n>0\}\\
u = g &\mbox{on }\{x_n=0\}.
\end{cases}\]
We then denote $E^+g := u$.

Similarly  $E^-g := -E^+(-g) $ is the unique solution, among functions satisfying \eqref{sumablegrowth},  of  the previous problem with $M^+$ replaced by $M^-$.
\end{prop}

\begin{proof}
Let $a_i = 2^{-i}\sup_{B_{2^{i+1}}^*\setminus B_{2^{i}}^*} |g|$ and
 \[ \phi(x) : =  \sup_{B_1^*}|g| +\sum_{i=0}^\infty 2^i a_i \phi_0(2^{-i}x),\]
By Lemma \ref{lemsupersol} we have $\phi\ge g$ in $x_n=0$ and $M^+\phi\le 0$ in $x_n>0$. On the other hand, using \eqref{control} we find
\[
\begin{split}
\sum_{j=0}^{\infty}2^{-j}\sup_{B_{2^j}} \phi
&\le \sum_{j=0}^{\infty} 2^{-j} \left( \sup_{B_1^*}|g| + C \sum_{i=0}^{j} 2^{i-j}a_i + C\sum_{j=0}^{\infty} \sum_{i=j}^\infty 2^{-j} a_i\right)  \\
&\le  2\sup_{B_1^*}|g| + C\sum_{i=0}^{\infty} \sum_{j=i}^{\infty} 2^{i-j}a_i + 2C \sum_{i=0}a_i
\\
& \le2\sup_{B_1^*}|g| + 4C \sum_{i=0}a_i <\infty.
\end{split}
\]
Thus in particular $\phi$ satisfies \eqref{sumablegrowth} with $u$ replaced by $\phi$.

Now we note that $\phi$ and $-\phi$ are respectively a supersolution and a subsolution of the problem $M^+ u= 0$ in $\{x_n>0\}$,  $u =g$ on $\{x_n=0\}$. Then, we can prove the existence of a continuous viscosity solution between $-\phi$ and $\phi$ in several standard ways.

One option is to choose any continuous extension $\bar g$ of $g$ to $\{x_n>0\}$ such that $|\bar g|\le \phi$ and to solve in large balls  $M^+ u_R= 0$ in $B_R^+$,  $u =\bar g$ in $\partial B_R^+$. Letting $R\uparrow \infty$ and using the stability of viscosity solutions under local uniform converge, we find a solution of the of the problem in all of $x_n>0$. The barriers $\pm\phi$ guarantee the convergence. Another option is to proof the existence of a solution in the half space directly by Perron's method.

The uniqueness of viscosity solution to this problem among continous functions $u$  satisfying \eqref{sumablegrowth} is a straightforward consequence of the maximum principle in Lemma \ref{lem-sumablegrowth} and the fact that the difference $w$ of two solutions satisfies $M^+ w\ge 0$ and $M^-w\le 0$ in $\{x_n>0\}$,  and $w=0$ on $\{x_n=0\}$.
\end{proof}

We next construct 1D solutions in $\R^2_+$.

\begin{prop}\label{prop-1D}
For any $\beta\in(0,1)$, let us consider the function $\varphi_\beta^\pm(x,y):=E^\pm(x_+)^\beta$ in $\R^2_+$.
Then,
\begin{itemize}
\item[(a)] We have
\[\partial_y\varphi_\beta^+=\overline C(\beta) x^{\beta-1}\qquad \textrm{in}\ \{x>0\}\cap\{y=0\},\]
\[\partial_y\varphi_\beta^-=\underline C(\beta) x^{\beta-1}\qquad \textrm{in}\ \{x>0\}\cap\{y=0\}.\]
The constants $\overline C$ and $\underline C$ depend only on $\beta$ and ellipticity constants.

\item[(b)] The functions $\overline C(\beta)$ and $\underline C(\beta)$ are continuous in $\beta$, and there are
\[0<\beta_1<\frac12<\beta_2<1\]
such that
\[\overline C(\beta_1)=0\qquad \textrm{and}\qquad \underline C(\beta_2)=0.\]
Moreover, $\beta_1$ and $\beta_2$ are unique.

\item[(c)] For any small $\delta>0$, we have
\[\frac12-\delta<\beta_1<\frac12<\beta_2<\frac12+\delta\quad\textrm{whenever}\quad |\Lambda-1|+|\lambda-1|\leq \delta/C,\]
with $C$ universal.
\end{itemize}
\end{prop}

We will need the following auxiliary result.

\begin{lem}\label{aux}
Let $w_k=E^+ g_k$  (resp. $w_k=E^- g_k$) where
\begin{equation} \label{unifgrowth}
\sum_{i\ge 1} 2^{-i}\sup_{B^*_{2^i} }\bigl|g_k\bigr| \le C,
\end{equation}
and
\begin{equation}\label{unifC2alpha}
\|g_k\|_{C^{1,\alpha}(\overline{B_{1/2}^*})} \le C,
\end{equation}
for some $\alpha\in(0,1)$, with $C$ independent of $k$.

Suppose that, for some $g\in C(\R^{n-1}\times\{0\})$
\begin{equation}\label{decayconvergence}
\sum_{i\ge 1} 2^{-i}\sup_{B_{2^i} }\bigl|g_k-g\bigr|  \rightarrow  0 \quad \mbox{on }\{x_n=0\}.
\end{equation}
Then, $|\partial_{x_n} w_k- \partial_{x_n}w|(0) \rightarrow 0$, where $w=E^+g$ (resp. $w=E^- g$).
\end{lem}

\begin{proof}
We first show that $w_k\in C^{1,\alpha}\big(\overline{B_{1/4}^+}\big)$, with a bound independent of $k$, and that $w_k\to w$ uniformly in $\overline{B_{1/4}^+}$.

Indeed, it follows from \eqref{unifgrowth} and from Lemma \ref{lemsupersol} (see also the proof of Proposition \ref{prop-extensions}) that $\|w_k\|_{L^\infty(B_1^+)}\le C$, with $C$ independent of $k$. Then, by the $C^{1,\alpha}$ estimates up to the boundary (see \cite{CC}) using \eqref{unifC2alpha} we obtain that  $\|w_k\|_{C^{1,\alpha}(\overline{B_{1/4}^+})} \le C$.

On the other hand, $w_k-w$ is a viscosity solution of $M^-(w_k-w) \le 0\le M^+(w_k-w)$ in $\{x_n>0\}$. Then by \eqref{decayconvergence} ---using again Lemma \ref{lemsupersol}--- we find $\sup_{B_1^+} (w_k-w) \rightarrow 0$.

Since all the $w_k$ are uniformly $C^{1,\alpha}(\overline{B_{1/4}^+})$ and converge uniformly to $w$ in $\overline{B_1^+}$ we find in particular $w_k\rightarrow w$ in $C^1(\overline{B_{1/4}^+})$. 
Thus, $|\partial_{x_n} w_k- \partial_{x_n}w|(0) \rightarrow 0$.
\end{proof}

We now give the:

\begin{proof}[Proof of Proposition \ref{prop-1D}]
(a) It follows by the scaling properties of $M^\pm$ and by uniqueness of $E^\pm$ that $\varphi_\beta^\pm$ are homogeneous functions of degree $\beta$.
Thus, part (a) follows, with
\[\overline C(\beta)=\partial_y\varphi_\beta^+(1,0),\qquad \underline C(\beta)=\partial_y\varphi_\beta^-(1,0).\]

(b) It follows from Lemma \ref{aux} ---translating the origin to the point (1,0)--- that $\partial_y \varphi_{\beta'}^\pm (1,0 )\to \partial_y \varphi_\beta^\pm(1,0)$.
As a consequence, $\overline C(\beta)$ and $\underline C(\beta)$ are continuous in $\beta\in(0,1)$. Although for $\beta =0$, the function $(x_+)^\beta= \chi_{x>0}$ has a discontinuity, we can easily adapt the proof of Lemma \ref{aux} to this situation by using that the only discontinuity point is at $(0,0)$ and that the solution is  bounded near this discontinuity point.

Note instead that a similar continuity property is not true as $\beta\uparrow 1$, since we approach the critical growth and hence we can not guarantee that $\|\varphi_\beta^\pm\|_{L^\infty(\overline{B_{1/4}^+(1,0))}}$ stays bounded as $\beta\uparrow 1$. In fact, we will show later on in this proof that this $L^\infty$ norm diverges.

Now, when $\beta=0$, as said abobe  $\varphi_\beta^\pm(x,0)= \chi_{\{x>0\}}$  and  Hopf lemma implies that $\partial_y \varphi_{\beta}^\pm (1,0 )<0$. Thus,
\[\lim_{\beta\downarrow 0} \underline C(\beta)\leq \lim_{\beta\downarrow 0} \overline C(\beta)<0.\]

On the other hand, we claim that
\begin{equation}\label{asbetato1}
\overline C(\beta)\geq \underline C(\beta) \ge \frac{c}{1-\beta} \to\infty\quad  \mbox{as}\quad \beta \uparrow 1.
\end{equation}
Indeed, let $\psi$ be the subsolution of Lemma \ref{lemsubsol}, with $r_0=\frac14$ and extended by zero outside $B_1$.
Consider the new subsolution
\[ \psi_k (x,y) = \sum_{i=0}^k 2^{\beta i}  \psi (2^{-i}x-\frac12,2^{-i}y),\]
which satisfies $M^-\psi_k\ge 0$ in all of $\{y>0\}$.

Note that, since $r_0=1/4$, the functions  we have $\psi (2^{-i}x-1/2,2^{-i}y)$ have disjoint supports at $y=0$.
Thus, we find
\[ \psi_k (x,0)  \le  2^{i\beta} \chi_{\{0<x< 2^i\}} \quad\mbox{for all }k \mbox{ and } i\]
In particular $2^{-\beta} \psi_k \le (x_+)^\beta$ on $\{y=0\}$.
Now, for fixed $\beta$ we readily show, using Lemma \ref{lem-sumablegrowth} and Proposition \ref{prop-extensions}, that
\begin{equation}\label{bbybelow}
2^{-\beta} \psi_k \le \varphi_\beta^{-} = E^- (x_+)^\beta \quad \mbox{(for all $k$)}.
\end{equation}

But note that, by Lemma \ref{lemsubsol}, at $x=\frac14$ we have $\psi_k(\frac14,0)=0$ and thus
\[  \psi_k \left(\frac14,y\right) = \sum_{i=0}^k  2^{(\beta-1) i} (\partial_i \psi) \left(2^{-i}-\frac12,2^{-i}y\right)  \ge c\,\frac{1-2^{(\beta-1)k}}{1-2^{\beta-1} } y.\]
for $|y|<1/2$. Letting $k\to\infty$, using  \eqref{bbybelow}, and recalling that $\varphi_\beta^{-}$ is homogeneous of degree $\beta$ we obtain
\[  \varphi_\beta^{-} (x,y) \ge \frac{c}{1-2^{\beta-1}} \,y \ge \frac{c}{1-\beta} \,y  \quad \mbox{for } x\in\left(\frac12,\frac32\right),\  y\in(0,1)\]
for some $c>0$ universal.

As $\beta\uparrow1$, thus $\varphi_\beta^{-} (x,y)$ is a nonnegative solution in $Q= (1,2,3/2)\times(0,1)$ with trace $x^\beta$ on $(1,2,3/2)\times \{y=0\}$ and that is arbitrarily large
in $(1,2,3/2)\times(1/2,1)$.
Then it is immediate to show that there is a quadratic polynomial $P$ satisfying $M^-P\ge0$ (subsolution), such that $P$ touches $\varphi_\beta^{-}$  by below in $\overline Q$  at the point $x=1$, $y=0$, and with $\partial_y P (1,0)$ arbitrarily large. Thus $\underline C(\beta)$ is arbitrarily large as $\beta \to 1$ ---with a growth $c/(1-\beta)$---, finishing the proof of the claim \eqref{asbetato1}.

Finally, as said before, $\overline C$ and $\underline C$ are continuous functions. Thus,  there are $0<\beta_1\leq \beta_2<1$ such that $\overline C(\beta_1)=0$ and $\underline C(\beta_2)=0$.

The uniqueness of the exponents $\beta_1$ and $\beta_2$ follows by a simple contact argument.
Indeed, if $\beta<\beta'$  then some translation (to the right) of the function $\varphi_\beta^{+}$ touches $\varphi_{\beta'}^{+}$ by below at some point on $\{x>0,y=0\}$. But since the two functions are homogeneous the sign of their vertical derivatives is the same on all of $\{x>0,y=0\}$. This shows that ${\rm sign}\bigl( \overline C(\beta')\bigr) > {\rm sign}\bigl( \overline C(\beta)\bigr)$,
where the strict inequality is a consequence of Hopf Lemma. This implies that the zero of  $\overline C$ is unique. The same argument applies to~$\underline C$.

Finally, using the same contact argument to compare $\varphi_{\beta_1}^+$ and $\varphi_{\beta_2}^-$ with the harmonic extension of $(x_+)^{1/2}$ (i.e. the solution for the Laplacian), we obtain $\beta_1<\frac 12<\beta_2$.

(c) Let $\psi$ be the solution of
\[\psi(x,0)=(x_+)^{\frac12-\delta}\quad\textrm{on}\quad \{y=0\},\]
\[\Delta \psi=-\kappa r^{-\frac32-\delta}\quad\textrm{in}\quad \{y>0\},\]
where $r=\sqrt{x^2+y^2}$.
Notice that $\psi$ is homogeneous of degree $\frac12-\delta$ in $\R^2$.

Notice also that when $\kappa=0$ then $\psi_y(x,0)=-c(\delta)x^{-\frac12-\delta}<0$ for $x>0$.
Thus, if $\kappa$ is small, we will have $\psi_y(x,0)\leq -\frac12 c(\delta)x^{-\frac12-\delta}$ for $x>0$.
In fact, a simple computation shows that $c(\delta)\geq c\delta$ for $\delta$ small.
Thus, by linearity, we may take $\kappa\geq c\delta>0$, too.

Let us now check that, if $|\Lambda-1|+|\lambda-1|\leq \gamma$, with $\gamma>0$ small, then
\[M^+\psi\leq0\quad \textrm{in}\quad \{y>0\}.\]
For this, notice that by homogeneity of $\psi$ we only need to check it on $\partial B_1$, where $\psi$ is $C^2$.
Also, notice that
\[M^+\psi=\lambda\Delta\psi+(\Lambda-\lambda)(\textrm{sum of positive eigenvalues of }D^2\psi),\]
so that
\[M^+\psi\leq \lambda\Delta\psi+C(\Lambda-\lambda)\leq -\lambda \kappa+C\gamma\leq -c\delta+C\gamma\leq0\]
provided that $\gamma\leq \delta/C$.

Thus,
\[M^+\psi\leq0\quad\textrm{on}\quad \{y>0\},\]
\[\psi_y\leq0 \quad\textrm{on}\quad \{y=0,\,x>0\},\]
\[\psi\textrm{ is homogeneous of degree }\frac12-\delta.\]
This, and the same contact argument as before, yields $\frac12-\delta<\beta_1$.
Repeating the same argument with $\frac12+\delta$, we get $\frac12+\delta>\beta_2$, and thus the proposition is proved.
\end{proof}

As a consequence, we have the following.

\begin{cor}\label{cor-1D}
Given $e\in S^{n-2}$, let
\[w_0^+(x):=\varphi_{\beta_1}^+(x'\cdot e,|x_n|),\qquad w_0^-(x):=\varphi_{\beta_2}^-(x'\cdot e,|x_n|),\]
where $\varphi_\beta^\pm$ and $\beta_1,\beta_2$ are given by Proposition \ref{prop-1D}.
Then,
\[\begin{cases}
M^\pm w_0^\pm=0\quad &\textrm{in}\ \R^n\setminus\bigl(\{x'\cdot e\leq 0\}\cap\{x_n=0\}\bigr)\\
w_0^\pm=0\quad &\textrm{on}\ \{x'\cdot e\leq 0\}\cap\{x_n=0\}.\end{cases}\]
The functions $w_0^+$ and $w_0^-$ are homogeneous of degree $\beta_1$ and $\beta_2$, respectively, and $0<\beta_1<\frac12<\beta_2<1$.

Moreover, $\frac12-\delta<\beta_1<\frac12<\beta_2<\frac12+\delta$ whenever $|\Lambda-1|+|\lambda-1|\leq \delta/C$.
\end{cor}

\begin{proof}
The result follows from Proposition \ref{prop-1D}, and taking into account that since $M^\pm w_0^\pm=0$ in $\{x_n\neq0\}$ and $w_0^\pm$ are $C^1$ at points on $\{x'\cdot e>0\}\cap\{x_n=0\}$, then they also solve the equation therein.
\end{proof}

\subsection{A maximum principle type Lemma}

We finally prove the following Lemma, similar to \cite[Lemma 5]{ACS}.

\begin{lem}\label{lem-max}
Let $c_0,c_1$ be given positive constants with $c_1 <  \sqrt{\lambda/(9n\Lambda)}$ ---i.e. universally small enough.
Then, there exists $\sigma>0$ for which the following holds.

Assume $v\in C(\overline{B_1})$ satisfies
\begin{itemize}
 \item $M^-v\leq \sigma$ in $B_1\setminus\Omega^*$, with $\Omega^*\subset\{x_n=0\}$
 \item $v=0$ on $\Omega^*$
 \item $v\geq c_0>0$ for $|x_n|\geq c_1>0$
 \item $v\geq -\sigma$ in $B_1$
\end{itemize}
Then, $v\geq0$ in $B_{1/2}$.
Moreover, $v\geq c_2 |x_n|$ in $B_{1/2}$, for some $c_2>0$ (small).
\end{lem}

\begin{proof}
Let us prove that $v\geq0$ in $B_{1/2}$.
Once this is proved, then $v\geq c_2|x_n|$ follows from the standard subsolution of Hopf's lemma ---see \eqref{Hopf}--- provided that $\sigma$ is small enough.

Assume there is $z=(z',z_n)\in B_{1/2}\cap\{|x_n|< c_1\}$ such that $v(z)<0$.
Let
\[Q=\left\{(x',x_n)\,:\,|x'-z'|\leq \frac13,\ |x_n|\leq c_1\right\}\]
and
\[P(x)=|x'-z'|^2-\frac{n\Lambda}{\lambda} x_n^2.\]
Notice that $M^+P=-\Lambda$.

Define
\[w=v+\delta P,\]
where $\delta>0$ is such that $0<C\sigma<\delta<c_0/C$, with $C$ large enough.
Then, we have
\begin{itemize}
\item $w(z)=v(z)-\delta\Lambda z_n^2<0$
\item $M^-w\leq M^-v+\delta\, M^+P\leq \sigma-\delta\Lambda\leq0$ outside $\Omega^*$
\item $w\geq0$ on $\Omega^*$
\end{itemize}
Thus, $w$ must have a negative minimum on $\partial Q$.

On $\partial Q\cap \{|x_n|=c_1\}$ we have
\[w\geq c_0-\delta \frac{n\Lambda}{\lambda}c_1^2\geq0.\]
On $\partial Q\cap \{|x'-z'|=1/3\}\cap \{0\leq |x_n|\leq c_1\}$, we have $v\geq -\sigma$, so that
\[w\geq -\sigma+ \delta\left(\frac19-\frac{n\Lambda}{\lambda} c_1^2\right)\geq0.\]
Hence, $w\geq0$ on $\partial Q$ and we have reached a contradiction.
Therefore, $v\geq0$ in $B_{1/2}$, as desired.
\end{proof}

\section{A boundary Harnack inequality}
\label{sec3}

We prove here a boundary Harnack inequality in ``slit'' cones, for solutions that are monotone in some ``outwards''  directions.
More precisely, we establish the following.

\begin{prop}\label{bdry-H}
Let $\Sigma^*\subset \R^{n-1}\times \{0\}$ be some nonempty closed convex cone satisfying
\begin{equation}\label{inclusion}
\Sigma^*\subset  \biggl\{ \frac{x}{|x|}\cdot e\le -\varepsilon \biggr\}
\end{equation}
for some $e\in S^{n-2}$ and $\varepsilon \in (0,1/8)$.
Let $\theta_1,\theta_2$ be unit vectors in $\R^{n-1}\times \{0\}$ with $-\theta_i\in \Sigma^*$.

Assume that $u_1,u_2\in C(B_1)$ satisfy
\begin{equation}\label{eqbddmq}
M^+(au_1+bu_2)\geq0\quad\textrm{in}\ B_1\setminus\Sigma^*
\end{equation}
for all $a,b\in\R$,
\[u_1= u_2=0\quad\mbox{on }B_1^*\cap  \Sigma^*.\]
Assume also $u_i\geq0$ in $B_1^+$, \,$\sup_{B_{\varepsilon/2}} u_1= \sup_{B_{\varepsilon/2}} u_2$, and $u_i$ is monotone nondecreasing in the direction $\theta_i$ in all of $B_1$ ---that is, $u_i(\bar x)\ge u_i(x)$ whenever $\bar x -x = t\theta_i$ for some $t\ge0$ and $x,\bar x \in B_1$.

Then,
\[ \frac{1}{C\varepsilon^{-M} }\,u_2\leq u_1\leq C\varepsilon^{-M}  \,u_2\qquad\textrm{in}\ \overline{B_{\varepsilon/4}},\]
where  $C$ and $M$ are positive universal constants.
\end{prop}

\begin{proof}
We may and do assume that
\begin{equation}\label{sup}
\sup_{B_{\varepsilon/2}} u_i=1.
\end{equation}

\vspace{3mm}

{\em Step 1.} We define
\[
A_\varepsilon :=  B_{7/8}\cap \{  x\cdot e  \ge \varepsilon/4\}.
\]
We first prove that  that
\begin{equation}\label{comp}
0<  C_\varepsilon^{-1} \le  \inf_{A_\varepsilon} u_i \le 1
\end{equation}
where $C_\varepsilon := C\varepsilon^{-M}$ for some positive universal constants $C$ and $M$. Thoughout the proof $C_\varepsilon$ denotes a constant of this form though $C$ and $M$ may vary from line to line.

Indeed, first note that by taking the four choices $a=\pm 1, b=0$ and $a=0, b = \pm 1$ in \eqref{eqbddmq} we obtain that $u_i$ are viscosity solutions of
\[ M^- u_i \le 0 \le M^+ u_i \quad \in B_1\setminus\Sigma^*.\]
Thus, using a standard chain of interior Harnack inequalities we have
\[ \sup_{A_\varepsilon} u_i \le C_\varepsilon \inf_{A_\varepsilon} u_i .\]

On the other hand, let us show that
\[ \mbox{given }x\in B_{\varepsilon/2} \mbox{ exist }\bar x \in A_\varepsilon, \ t\ge 0 \mbox{ such that } \bar x -x = t\theta_i\]

Indeed, if $x\in B_{\varepsilon/2}$ we have $x\cdot e> -\varepsilon/2$ and thus, using \eqref{inclusion} the point $\bar x =x + \frac 3 4\,\theta_i$ satisfies
\[ \bar x\cdot e \ge -\varepsilon/2 + 3\varepsilon /4 \ge \varepsilon/4.\]
Here we have used that $\theta_i\cdot e\ge \varepsilon$ since $-\theta_i$ are unit vectors in $\Sigma^*$ and we have \eqref{inclusion}.
In addition, $\bar x\in B_{7/8}$ since $\left| \frac 3 4\,\theta_i\right| =3/4$ and $|x|= \varepsilon/2\le 1/8$.

Thus, using the monotonicity of $u_i$ in the direction $\theta_i$ we have that
\[ 1=  \sup_{B_{\varepsilon/2}} u_i \le \sup_{A_\varepsilon} u_i  \le C_\varepsilon \inf_{A_\varepsilon} u_i \le C_\varepsilon \sup_{B_{\varepsilon/2}} u_i = C_\varepsilon,\]
where for the last inequality we have used that $A_\varepsilon \cap B_{\varepsilon/2}\neq \emptyset$.

Thus, \eqref{comp} follows.

\vspace{3mm}

{\em Step 2.} We next prove that, with $C_\varepsilon$ as above,
\begin{equation}\label{comparabilB1/4}
u_1\ge  C_\varepsilon^{-1} u_2 \quad \mbox{in } B_{\varepsilon/4}^*.
\end{equation}

We consider the rescaled solutions $\bar u_i(x) = u_i\left(\frac\varepsilon 2 \,x\right)$.
Then, $\bar u_1,\bar u_2\in C(B_1)$ satisfy
\begin{equation}\label{eqbddmqRES}
M^+(a \bar u_1+b \bar u_2)\geq0\quad\textrm{in}\ B_2\setminus\Sigma^*
\end{equation}
for all $a,b\in\R$, and
\[\bar u_1= \bar u_2=0\quad\mbox{on }B_2^*\cap  \Sigma^*.\]

In addition we have $\bar u_i\geq0$ in $B_2$.
$\sup_{B_{1}} \bar u_i=1$  ---recall \eqref{sup}---, and, by Step 1,
\[ C_\varepsilon^{-1} \le  \inf_{B_1\cap \{e\cdot x \ge 1/4\} } \bar u_i . \]

Using again a chain of interior  Harnack inequalities we obtain
\begin{equation}\label{bybelowB*}
 C_\varepsilon^{-1} \le  \inf_{B^*} \bar u_i ,
 \end{equation}
where $B^* = B^*_{1/4}(z)$ for $z=e/2$.

Fix $\rho=1/10$. Let $\eta\in C^2(\overline {B_{1}})$ be some smooth ``cutoff'' function with $\eta=1$ for $|x|\ge 1-\rho$ and $\eta=0$ in $B_{1/2}$. Let us call
\[C_1:=  \sup_{B_{1}} M^+ \eta = \sup_{B_{1-\rho}} M^+ \eta  \]

Let $\phi$ be the subsolution of Lemma \ref{lemsubsol} ---with $\rho = 1/10$ and $B^* = B^*_{1/4}(z)$ for $z=e/2$, as before.

We will show next that, for $C_\varepsilon\ge 1$ large enough,
\begin{equation}\label{done}
 C_\varepsilon \bar u_1 + \eta \ge \bar u_2 + C_1 \phi \quad \mbox{in }B_1.
\end{equation}

Indeed, on the one hand since $0\le \bar  u_i \le 1$ in $B_1$ we and $\eta= 1$ for for $|x|\ge 1-\rho$  we have and $\phi=0$ on $\partial B_1$ we have that \eqref{done} holds on $\partial B_1$. On the other hand we have
\[
M^- ( C\bar  u_1 + \eta - \bar u_2 - C_1 \phi ) \le M^+\eta -C_1 M^- \phi \le C_1\chi _{B_{1-\rho}}   - C_1\chi _{B_{1-\rho}}  \le 0 \quad \mbox{in }B_1\setminus B^*
 \]
while, using \eqref{bybelowB*}
\[
C_\varepsilon \bar u_1 + \eta - \bar u_2 - C_1 \phi \ge (C_\varepsilon \bar u_1  - \bar u_2) + (C_\varepsilon \bar u_1 -  C_1\phi ) \ge 0 \quad \mbox{ in } B^*
\]
where we recall that $C$ is a constant of the type $C\varepsilon^{-M}$ with $C$ and $M$ universal and varying from line to line.

Thus, \eqref{done} follows using by the maximum principle. Finally, since $\phi\ge 0$ and $\eta =0$ in $B_{1/2}$ from \eqref{done} we deduce that
\[ C_\varepsilon \bar u_1 \ge  \bar u_2 \quad \mbox{in }B_{1/2}\]
and thus after rescaling we obtain \eqref{comparabilB1/4}.

Finally, since the roles of $\bar u_1$ and $\bar u_2$ are interchangeable we obtain the comparability of $\bar u_1$ and $\bar u_2$ in $\overline {B_{1/8}^+} $.
Rescaling back, we obtain that $u_1$ and $u_2$ are comparable in $B_{\varepsilon/4}$, as desired.
\end{proof}

As a consequence we obtain the following.

\begin{cor}\label{bdry-H-allspace}
Let $\Sigma^*\subset \R^{n-1}\times \{0\}$ be some nonempty closed convex cone satisfying
\begin{equation}\label{inclusion}
\Sigma^*\subset  \biggl\{ \frac{x}{|x|}\cdot e\le -\varepsilon \biggr\}
\end{equation}
for some $e\in S^{n-2} $ and $\varepsilon \in (0,1/8)$.
Let $\theta_1,\theta_2$ be unit vectors in $\R^{n-1}\times \{0\}$ with $-\theta_i\in \Sigma^*$.

Assume that $u_1,u_2\in C(\R^n)$ satisfy
\begin{equation}\label{eqbddmq}
M^+(au_1+bu_2)\geq0\quad\textrm{in}\ \R^n\setminus\Sigma^*
\end{equation}
for all $a,b\in\R$,
\[u_1= u_2=0\quad\mbox{on } \Sigma^*.\]
Assume also $u_i\geq0$ in $\R^n$, \,$\sup_{B_{1}} u_1= \sup_{B_{1}} u_2$, and $u_i$ is monotone nondecreasing in the direction $\theta_i$ in all of $\R^n$ ---that is, $u_i(\bar x)\ge u_i(x)$ whenever $\bar x -x = t\theta_i$ for some $t\ge0$ .

Then,
\[ \frac{1}{C\varepsilon^{-M} }\,u_2\leq u_1\leq C\varepsilon^{-M}  \,u_2\qquad\textrm{in all of }\ \R^n.\]
where $C$ and $M$ are positive universal constants.
\end{cor}

\begin{proof}
We may assume that $\sup_{B_{1/2}} u_1= \sup_{B_{1/2}} u_2=1$.

Let $R\ge 4$ arbitrary.
Consider the two rescaled functions $\bar u_1$ and $\bar u_2$ defined by
\[ \bar u_i (x) = \frac{u_i(Rx)}{C_i}\quad \mbox{for }C_i = \|u_i\|_{L^\infty(B_R)}.\]
By Proposition \ref{bdry-H} we obtain that
\[ C_\varepsilon^{-1} \bar u_2\leq \bar u_1\leq C_\varepsilon\, \bar u_2\qquad\textrm{in}\ \overline{B_{1/8}},\]
where $C_\varepsilon = C\varepsilon^{-M}$ with $C$ and $M$ universal constants.

Thus, using that
\[  1=  \|u_i\|_{L^\infty(B_{1/2})}  =  C_i\| \bar u_i\|_{L^\infty(B_{1/(2R)})} \]
Since we have that  $\| \bar u_1\|_{L^\infty(B_{1/(2R)})}$ and  $\| \bar u_2\|_{L^\infty(B_{1/(2R)})}$ are comparable (recall that $R\ge 4$) we obtain that $C_1$ and $C_2$ are comparable and thus, scaling back, that
\[ C_\varepsilon^{-1}  u_2\leq u_1\leq C_\varepsilon\,  u_2\qquad\textrm{in}\ \overline{B_{R/8}}.\]
Since $R$ can be taken arbitrarily large the Corollary follows.
\end{proof}

\section{Global solutions}
\label{sec4}

In this Section we prove that any global solution to the obstacle problem with subquadratic growth must be 1D on $\{x_n=0\}$.

\begin{thm}\label{thmclassif}
Let $F$ be as in \eqref{F}, and $u\in C(\R^n)$ be any viscosity solution of
\begin{equation}\label{eqnliouville}
\begin{cases}
F(D^2u)\leq 0\quad &\mbox{in } \R^n\\
F(D^2u) = 0  \quad &\mbox{in } \R^n\setminus\Omega^*\\
u=0 \quad &\mbox{on } \Omega^*\\
u\geq0 \quad &\mbox{on } \{x_n=0\},
\end{cases}\end{equation}
with
\begin{equation}\label{grad}
u(0)=0,\qquad \nabla u(0)=0.
\end{equation}
Assume that $u$ satisfies the following growth control
\begin{equation}\label{growth}
\|u\|_{L^\infty(B_R)} \le R^{2-\epsilon} \quad \mbox{ for all }R\ge 1.
\end{equation}
Then, either $u \equiv 0$, or
\[  u(x)=u_0(e\cdot x',\,x_n) \qquad \textrm{and}\qquad  \{u(x',0)=0\} =\{e\cdot x'\le 0\}\]
for some $e\in S^{n-2}$.
Moreover, $u_0$ is convex in the $x'$ variables.
\end{thm}

We will need the following intermediate steps in the proof of Theorem \ref{thmclassif}.

\begin{lem}\label{c-1}
Let $F$ be as in \eqref{F}, and $u\in C(\R^n)$ be any viscosity solution of
\[F(D^2u) = 0  \quad \mbox{in } \R^n,\]
with $u(0)=0$ and $\nabla u(0)=0$.
Assume that $u$ satisfies the growth control \eqref{growth}.
Then, $u\equiv0$.
\end{lem}

\begin{proof}
By interior $C^{1,1}$ estimates \cite{CC} ---here we use the convexity of the operator--- we have
\[\|D^2u\|_{L^\infty(B_1)}\leq C.\]
Applying the same estimate to the rescaled function $u(Rx)/R^{2-\epsilon}$, we find
\[\|D^2u\|_{L^\infty(B_R)}\leq CR^{-\epsilon},\]
for any $R\geq1$.
Letting $R\to\infty$, we deduce that $u$ is affine.
Since $u(0)=0$ and $\nabla u(0)=0$, it must be $u\equiv0$.
\end{proof}

We next prove the following.

\begin{prop}\label{c-2}
Let $F$ be as in \eqref{F}, and $u\in C(\R^n)$ be any viscosity solution of \eqref{eqnliouville}-\eqref{grad}-\eqref{growth} which is convex in the $x'=(x_1,...,x_{n-1})$ variables.

Assume in addition that $\Sigma^*=\{u=0\}\cap \{x_n=0\}$ is a closed convex cone with nonempty interior and vertex at the origin.
Then, either $u\equiv0$ or
\[\Sigma^*=\{x'\cdot e\leq 0\}\]
for some $e\in S^{n-2}$.
\end{prop}

\begin{proof}
Assume that $u$ is not identically zero and that $\Sigma^*$ is not a half-space.

Notice that if $\Sigma^*$ contains a line $\{te'\,:\, t\in\R\}$ then by convexity of $u$ we will have $u(x+te')=u(x)$ for all $t\in\R$, $x\in \R^n$.
Hence, if $\Sigma^*$ contains a line, $u$ is a solution in dimension $n-1$.
Therefore, by reducing the dimension $n$ if necessary, we may assume that $\Sigma^*$ contains no lines.

In particular,
\begin{equation}\label{cone-eps}
\Sigma^*\subset \left\{ \frac{x'}{|x'|}\cdot e\leq -\varepsilon\right\}
\end{equation}
for some $e\in S^{n-2}$ and some $\varepsilon>0$.

Let $\varepsilon>0$ be the largest positive number for which \eqref{cone-eps} holds.
Let $e_1\in S^{n-2}$ be such that $-e_1\in \Sigma^*$ and $-e_1\cdot e=-\varepsilon$.

Since $-e\in \Sigma^*$ and $-e_1\in \Sigma^*$, then by convexity of $u$ we have
\[w=\partial_e u\geq0\quad\textrm{and}\quad w_1=\partial_{e_1}u\geq0\quad \textrm{on}\quad \{x_n=0\}.\]
Moreover, since $\Sigma^*$ contains no lines, then these two functions are positive in $\{x_n=0\}\setminus\Sigma^*$.
Moreover, we have
\[M^+(aw+bw_1)\geq 0\quad \textrm{in}\quad \R^n\setminus\Sigma^*\]
for all $a,b\in \R$.
Furthermore, the convexity of $u$ and the growth control \eqref{growth} yield
\[\|w\|_{L^\infty(B_R)}+\|w_1\|_{L^\infty(B_R)}\leq CR^{1-\epsilon}.\]
By the maximum principle in Lemma \ref{lem-sumablegrowth}, this implies
\[w=\partial_e u\geq0\quad\textrm{and}\quad w_1=\partial_{e_1}u\geq0\quad \textrm{in}\quad \R^n.\]
Therefore, by the boundary Harnack type principle in Corollary \ref{bdry-H-allspace}, this means that
\[\partial_{e_1}u\geq c\partial_e u\quad \textrm{in}\quad \R^n.\]
Equivalently, $\partial_{e_1-ce}u\geq0$.
But then this yields $-(e_1-ce)\in\Sigma^*$, which combined with $-(e_1-ce)\cdot e=-\varepsilon-c$ is a contradiction with \eqref{cone-eps}.
\end{proof}

Using Lemma \ref{c-1} and Proposition \ref{c-2}, we can now give the:

\begin{proof}[Proof of Theorem \ref{thmclassif}]
If $u\equiv 0$ there is nothing to prove.
By the (local) semiconvexity estimates in \cite{Fer} applied (rescaled) to a sequence of balls with radius converging to infinity, we readily prove $u$ is convex in the $x'$ variables.
Thus, $\Omega^*$ is convex.

If $\Omega^*=\{x'\cdot e\leq 0\}$ for some $e\in S^{n-2}$, then by convexity we have $u(x',0)=u_0(x'\cdot e,0)$, and thus $u(x)=u_0(x'\cdot e,x_n)$, where $u_0$ is a 2D solution to the problem.

We next prove that if $\Omega^*$ is not a half-space, then there is no solution $u$.

Assume by contradiction that $\Omega^*$ is not a half-space and that $u$ is a nonzero solution.
Then, we do a blow-down argument, as follows.

For $R\ge 1$ define
\[ \theta(R)  = \sup_{R'\ge R} \frac{ \|u\|_{L^\infty(B_{R'})} }{ (R')^{2-\epsilon}}.\]
Note that $0<\theta(R)<\infty$ and that it is nonincreasing.

For all $m\in \mathbb N$  there  is $R'_m\ge m$ such that
\[ \ {(R'_m)}^{\epsilon-2} \|u_m\|_{L^\infty(B_R)} \ge \frac{\theta(m)}{2} \ge \frac{\theta(R'_m)}{2}.\]
Then the blow down sequence
\[ u_m(x) := \frac{u(R'_m x)}{(R'_m)^{2-\epsilon}\theta(R'_m)}\]
satisfies the growth control
\[   \|u_m\|_{L^\infty(B_{R})} \le R^{2-\epsilon}\quad \mbox{for all }R\ge 1 \]
and also
\[\|u_m\|_{L^\infty(B_1)} \ge \frac 12.\]

By $C^{1,\alpha}$ estimates \cite{Fer} and the Arzel\`a-Ascoli theorem, the sequence $u_m$ converges (up to a subsequence) locally uniformly in $C^1$ to a function $u_\infty$ satisfying
\begin{equation}\label{growthuinfty}
\|u_\infty\|_{L^\infty(B_{R})} \le R^{2-\epsilon}\quad \mbox{for all }R\ge 1,
\end{equation}
\begin{equation}\label{gnondeguinfty}
\|u_\infty\|_{L^\infty(B_1)}  \ge \frac 12,
\end{equation}
and
\begin{equation}\begin{cases}
F(D^2 u_\infty) = 0  \quad &\mbox{in } \R^n\setminus \Sigma^*\\
F(D^2 u_\infty) \leq 0  \quad &\mbox{in } \R^n\\
D^2 u_\infty \ge 0 \quad &\mbox{in } \R^n\\
u_\infty = 0 &\mbox{in }\Sigma^*,
\end{cases}\end{equation}
where $\Sigma^*$ is the blow-down of the convex set $\Omega^*$.
Notice that, by convexity, since $\Omega^*$ was not a half-space, then $\Sigma^*$ is not a half-space.

If $\Sigma^*$ has nonempty interior, by Proposition \ref{c-2} there is no solution $u$.
If $\Sigma^*$ has empty interior, then by $C^{1,\alpha}$ regularity of $u$ we get $u_{x_n}=0$ in all of $\{x_n=0\}$.
But using Lemma \ref{c-1}, this yields $u\equiv0$ as well.

Thus, if $\Omega^*$ is not a half-space there is no nonzero solution $u$, as claimed.
\end{proof}

We also prove the following.

\begin{cor}\label{corclassif}
Let $F$ be as in \eqref{F}, and $\beta_1\in(0,\frac12)$ be given by Corollary \ref{cor-1D}.
Let $u\in C(\R^n)$ be any viscosity solution of \eqref{eqnliouville} satisfying \eqref{grad} and
\begin{equation}\label{growth-beta}
\|u\|_{L^\infty(B_R)} \le R^{1+\beta} \quad \mbox{ for all }R\ge 1,
\end{equation}
with $\beta<\beta_1$.
Then, $u \equiv 0$.
\end{cor}

\begin{proof}
By Theorem \ref{thmclassif}, we know that $u(x)=u_0(x'\cdot e,x_n)$, with $u_0$ convex in the first variable and vanishing on  $\{x_1\leq 0\}\cap\{x_2=0\}$.
Thus, we only need to prove the result in dimension $n=2$.
We denote $v=\partial_{x_1}u\geq0$ in $\R^2$.
Notice that
\[\begin{cases}
M^+ v\geq 0,\quad M^-v\leq 0\quad &\textrm{in}\ \R^2\setminus\bigl(\{x_1\leq 0\}\cap\{x_2=0\}\bigr)\\
v=0\quad &\textrm{on}\ \{x_1\leq 0\}\cap\{x_2=0\}.
\end{cases}\]
Notice also that, by convexity and \eqref{growth-beta}, we have $\|v\|_{L^\infty(B_R)}\leq CR^{\beta}$.

We now use the supersolution given by Corollary \ref{cor-1D}.
Indeed, let $w=w_0^+$ be the homogeneous function of degree $\beta_1$ satisfying
\[\begin{cases}
M^+ w=0\quad &\textrm{in}\ \R^2\setminus\bigl(\{x_1\leq 0\}\cap\{x_2=0\}\bigr)\\
w=0\quad &\textrm{on}\ \{x_1\leq 0\}\cap\{x_2=0\}.
\end{cases}\]
Then, using interior Harnack inequality, a simple application of the maximum principle yields
\[0\leq v\leq Cw\quad\textrm{in}\quad B_2\setminus B_1.\]
Here, we used that $\|v\|_{L^\infty(B_3)}\leq C$.
By comparison principle, we deduce
\[0\leq v\leq Cw\quad \textrm{in}\quad B_2.\]

Repeating the same argument at all scales $R\geq1$ ---using the rescaled functions $R^{-\beta_1}w(Rx)=w(x)$ and $R^{-\beta_1}v(Rx)$---, we find
\[0\leq v\leq CR^{\beta-\beta_1}w\quad \textrm{in}\quad B_{2R}\setminus B_R.\]
Here, we used that $\|v\|_{L^\infty(B_{3R})}\leq CR^{\beta}$.

By comparison principle, the previous inequality yields
\[0\leq v\leq CR^{\beta-\beta_1}w\quad \textrm{in}\quad B_R,\]
and thus letting $R\to\infty$ we find $v\equiv0$.
This means that $u(x_1,x_2)=\psi(x_2)$, for some function $\psi$.
But since $F(D^2u)=0$ in $\{x_2>0\}$ and in $\{x_2<0\}$, then $u(x_1,x_2)=ax_2$, and since $\nabla u(0)=0$, then $u\equiv0$, as desired.
\end{proof}

\section{Regular points and blow-ups}
\label{sec5}

We start in this section the study of free boundary points.
For this, we use some ideas from \cite{CRS}.

After a translation, we may assume that the free boundary point is located at the origin.
Moreover, by subtracting a plane, we may assume that
\[u(0)=0\quad\textrm{ and }\quad \nabla u(0)=0.\]
Moreover, we assume
\[\|u\|_{L^\infty(B_1)}=1,\qquad \|\varphi\|_{C^{1,1}}\leq1.\]

We say that a free boundary point is regular whenever (ii) in Theorem \ref{thm1} does \emph{not} hold, that is:

\begin{defi}
We say that $0\in\partial\{u=\varphi\}$ is a \emph{regular} free boundary point if
\[\limsup_{r\downarrow0}\frac{\|u\|_{L^\infty(B_r)}}{r^{2-\epsilon}}=\infty\]
for some $\epsilon>0$.
We say that it is a regular point with exponent $\epsilon$ and modulus $\nu$ if
\[\sup_{\rho\leq r\leq 1}\frac{\|u\|_{L^\infty(B_r)}}{r^{2-\epsilon}}\geq \nu(\rho)\]
where $\nu(\rho)$ is a given nonincreasing function satisfying $\nu(\rho)\to\infty$ as $\rho\downarrow0$.
\end{defi}

The main result of this section is the following.

\begin{prop}\label{prop-rescalings}
Assume that $0$ is a regular free boundary point with exponent $\epsilon$ and modulus $\nu$.
Then, given $\delta>0$, there is $r>0$ such that the rescaled function
\[ v(x) : = \frac{u(r x)}{\|u\|_{L^\infty(B_r)}} \]
satisfies
\begin{equation}
\bigl| v - u_0\bigr| + \bigl|\nabla v - \nabla u_0\bigr| \leq \delta \quad \mbox{in } B_1,
\end{equation}
for some global convex solution $u_0$ of \eqref{eqnliouville}-\eqref{grad}-\eqref{growth}, with $\|u_0\|_{L^\infty(B_1)}=1$.
The constant $r$ depends only on $\delta$, $\epsilon$, $\nu$, $n$, and $\lambda, \Lambda$.
\end{prop}

To prove this, we need the following intermediate step.

\begin{lem}\label{lem-compactness}
Given $\delta>0$, there is $\eta=\eta(\delta,\epsilon,n,\lambda,\Lambda)>0$ such that the following statement holds.

Let $\varphi$ be such that $\|\varphi\|_{C^{1,1}}\leq \eta$, and let $v\geq0$ be a function satisfying $v(0)=0$, $\nabla v(0)=0$,
\begin{equation}\label{comp1}
\begin{array}{rclll}
F(D^2v)&=&0\,&\textrm{in} &B_{1/\eta}\setminus \{x_n=0\}\\
\min(-F(D^2v),\,v-\varphi)&=& 0\,&\textrm{on} &B_{1/\eta}\cap \{x_n=0\},
\end{array}
\end{equation}
and
\begin{equation}\label{comp3}
\|v\|_{L^\infty(B_1)}=1,\qquad \|v\|_{L^\infty(B_R)}\leq CR^{2-\epsilon}\quad \textrm{for}\quad 1\leq R\leq 1/\eta.
\end{equation}
Then,
\[\bigl| v - u_0\bigr| + \bigl|\nabla v - \nabla u_0\bigr| \leq \delta \quad \mbox{in } B_1,\]
for some global convex solution $u_0$ of \eqref{eqnliouville}-\eqref{grad}-\eqref{growth}, with $\|u_0\|_{L^\infty(B_1)}=1$.
\end{lem}

\begin{proof}
The proof is by a compactness.
Assume by contradiction that for some $\delta>0$ we have a sequence $\eta_k\to0$, fully nonlinear convex operators $F_k$ with ellipticity constants $\lambda,\Lambda$, obstacles $\varphi_k$ with $\|\varphi_k\|_{C^{1,1}}\leq \eta_k$, and functions $v_k\geq 0$ satisfying $v_k(0)=0$, $\nabla v_k(0)=0$, \eqref{comp1}, and \eqref{comp3}, but such that
\begin{equation}\label{contr}
\bigl\| v_k - u_0 \bigr\|_{C^1(B_1)} \ge \delta \quad \mbox{for all global solution } u_0\quad\textrm{with}\quad \|u_0\|_{L^\infty(B_1)}=1.
\end{equation}
By the estimates in \cite{Fer,MS}, we have that $v_k$ are $C^{1,\alpha}$ in $B_R$, $R<1/\eta_k$, with an estimate
\[ \|v_k\|_{C^{1,\alpha}(B_R)} \le C(R) \quad \mbox{for all}\quad 1\leq R\leq 1/2\eta_k.\]
Thus, up to taking a subsequence, the operators $F_k$ converge (locally uniformly as Lipchitz functions of the Hessian) to some fully nonlinear convex operator $F$ with ellipticity constants $\lambda,\Lambda$. Likewise,  the functions $v_k$ converge in $C^1_{\rm loc}(\R^n)$ to a function $v_\infty$, which by stability of viscosity solutions ---see \cite{CC}--- is a global convex solution to the obstacle problem \eqref{eqnliouville} and satisfying \eqref{grad} and~\eqref{growth}.

By the classification result Theorem \ref{thmclassif}, we have
\[ v_\infty \equiv u_0,\quad \mbox{for some global solution}\ u_0.\]
Moreover, by \eqref{comp3} we have
\[\|u_0\|_{L^\infty(B_1)}=\|v_\infty\|_{L^\infty(B_1)}=1.\]

We have shown that $v_k\rightarrow u_0$ in the $C^1$ norm, uniformly on compact sets.
In particular, \eqref{contr} is contradicted for large $k$, and thus the lemma is proved.
\end{proof}

To prove Proposition \ref{prop-rescalings} we will also need the following.

\begin{lem}\label{simple-lem}
Assume $w\in L^\infty(B_1)$ satisfies $\|w\|_{L^\infty(B_1)}=1$, and
\[\sup_{\rho\leq r\leq 1}\frac{\|w\|_{L^\infty(B_r)}}{r^{2-\epsilon}}\geq \nu(\rho)\to\infty\quad \textrm{as}\ \rho\to0.\]
Then, there is a sequence $r_k\downarrow0$ for which $\|w\|_{L^\infty(B_{r_k})}\geq \frac12r_k^\mu$, and for which the rescaled functions
\[w_k(x)=\frac{w(r_k x)}{\|w\|_{L^\infty(B_{r_k})}}\]
satisfy
\[|w_k(x)|\leq C(1+|x|^\mu)\quad\textrm{in}\ B_{1/r_k},\]
with $C=2$.
Moreover, we have
\[ 0<1/k \le r_k\leq (\nu(1/k))^{-1/\mu}.\]
\end{lem}

\begin{proof}
Let
\[\theta(\rho):= \sup_{\rho\leq r\leq 1} r^{-\mu}\|w\|_{L^\infty(B_r)}.\]
By assumption, we have
\[ \theta(\rho)\geq\nu(\rho)\rightarrow \infty\qquad \textrm{as}\quad \rho\downarrow0.\]
Note that $\theta$ is nonincreasing.

Then, for every $k \in \mathbb N$ there is  $r_k\ge \frac1k$ such that
\begin{equation}\label{nondeg-m}
 (r_k)^{-\mu} \|w\|_{L^\infty(B_{r_k})} \ge \frac 12 \theta(1/k)\ge \frac 12 \theta(r_k).
\end{equation}
Note that since $\|w\|_{L^\infty(B_1)}=1$ then
\[(r_k)^{-\mu}\geq \frac12 \theta(1/k)\geq \frac12\nu(1/k),\]
and hence
\[0<1/k \leq r_k\leq (\nu(1/k))^{-1/\mu}.\]
Moreover, we have $\theta(r_k)\geq1$, and thus $\|w\|_{L^\infty(B_{r_k})}\geq \frac12r_k^\mu$.

Finally, by definition of $\theta$ and by \eqref{nondeg-m}, for any $1\leq R\leq 1/r_k$ we have
\[\|w_k\|_{L^\infty(B_R)}=\frac{\|w\|_{L^\infty(B_{r_kR})}}{\|w\|_{L^\infty(B_{r_k})}}
\leq \frac{\theta(r_kR)(r_kR)^\mu}{\frac12(r_k)^\mu \theta(r_k)}\leq 2R^\mu.\]
In the last inequality we used the monotonicity of $\theta$.
\end{proof}

We now give the:

\begin{proof}[Proof of Proposition \ref{prop-rescalings}]
Let $r_k\to0$ be the sequence given by Lemma \ref{simple-lem} (with $\mu=2-\epsilon$).
Then, the functions
\[u_k(x)=\frac{u(r_k x)}{\|u\|_{L^\infty(B_{r_k})}}\]
satisfy
\[|u_k(x)|\leq C(1+|x|^\mu)\quad\textrm{in}\ B_{1/r_k},\]
and
\begin{equation}\label{nondegvk}
\|u_k\|_{L^\infty(B_1)}=1,\qquad u_k(0)=0,\qquad \nabla u_k(0)=0.
\end{equation}
Moreover, they are solutions to the obstacle problem in $B_{1/r_k}$, i.e.,
\[\begin{array}{rclll}
F(D^2u_k)&=&0\,&\textrm{in} &B_{1/r_k}\setminus \{x_n=0\}\\
\min(-F(D^2u_k),\,u_k-\varphi_k)&=& 0\,&\textrm{on} &B_{1/r_k}\cap \{x_n=0\},
\end{array}\]
where
\[ \|\varphi_k\|_{C^{1,1}}=\frac{ \|\varphi(r_k\,\cdot\,)\|_{C^{1,1}}}{\|u_k\|_{L^\infty(B_{r_k})}}\leq \frac{C(r_k)^2}{(r_k)^{2-\epsilon}}=C(r_k)^\epsilon\]
converges to 0 uniformly as $k\to\infty$.
Therefore, by Lemma \ref{lem-compactness} for $k$ large enough (so that $(r_k)^\epsilon\leq (\nu(1/k))^{-1/(2-\epsilon)}\leq \eta$) we have
\[\bigl| v - u_0\bigr| + \bigl|\nabla v - \nabla u_0\bigr| \leq \delta \quad \mbox{in } B_1,\]
for some global convex solution $u_0$ of \eqref{eqnliouville}-\eqref{grad}-\eqref{growth}, with $\|u_0\|_{L^\infty(B_1)}=1$, as desired.
\end{proof}

\section{Lipschitz regularity of the free boundary}
\label{sec6}

We now prove that the free boundary is Lipschitz in a neighborhood of any regular point $x_0$.

\begin{prop}\label{prop-Lip}
Assume that $0$ is a regular free boundary point with exponent $\epsilon$ and modulus $\nu$.
Then, there exists $e\in S^{n-1}\cap \{x_n=0\}$ such that for any $\ell>0$ there exists $r>0$ for which
\[\partial_\tau u\geq0\quad \textrm{in}\ B_r\qquad \textrm{for all}\quad \tau\cdot e\geq \frac{\ell}{\sqrt{1+\ell^2}},\quad \tau\in S^{n-1}\cap \{x_n=0\}.\]
In particular, the free boundary is Lipschitz in $B_r$, with Lipschitz constant $\ell$.

The constant $r$ depends only on $\ell$, $\epsilon$, $\nu$, $n$, $\lambda$, $\Lambda$.
\end{prop}

To prove this, we need the following.

\begin{lem}\label{u0}
Let $u_0(x)=u_0(x'\cdot e,x_n)$ be a global solution of \eqref{eqnliouville}-\eqref{grad}-\eqref{growth}, with $\|u_0\|_{L^\infty(B_1)}=1$.
Let $\tau\in S^{n-1}\cap \{x_n=0\}$ be such that $\tau\cdot e>0$.

Then, for any given $\eta>0$ we have
\[\partial_\tau u_0\geq c_0(\tau\cdot e)>0\quad \textrm{in}\ \{x'\cdot e\geq\eta>0\}\cap B_2\]
and
\[\partial_\tau u_0\geq c_0(\tau\cdot e)>0\quad \textrm{in}\ \{|x_n|\geq\eta>0\}\cap B_2,\]
with $c_0$ depending only on $\eta$ and ellipticity constants.
\end{lem}

\begin{proof}
Since $u_0(x)=u_0(x'\cdot e,x_n)$ it suffices to show the result in dimension $n=2$.
In that case, we have $F(D^2u_0)=0$ in $\R^2\setminus\{x_1\leq0\}$, and satisfies $\partial_{x_1x_1}u_0\geq0$, $\partial_{x_1}u_0\geq0$ in $\R^2$.
Then, by the interior Harnack inequality, and using $\|u_0\|_{L^\infty(B_1)}=1$, it follows that
\[\partial_{x_1} u_0\geq c>0\quad \textrm{in}\ \{x_1\geq\eta>0\}\cap B_2\]
and
\[\partial_{x_1} u_0\geq c>0\quad \textrm{in}\ \{|x_2|\geq\eta>0\}\cap B_2,\]
as desired.
\end{proof}

We can now give the:

\begin{proof}[Proof of Proposition \ref{prop-Lip}]
Let $r>0$ be as in the proof of Proposition \ref{prop-rescalings}, and
\[v(x)=\frac{u(rx)}{\|u\|_{L^\infty(B_r)}}.\]
Then, $v$ satisfies
\[F(D^2v)=0\quad \textrm{in}\quad B_2\setminus \{x_n=0\},\]
\[\min\bigl(-F(D^2v),\, v-\varphi_r\bigr)=0\quad \textrm{on}\quad B_2\cap\{x_n=0\}.\]
Moreover, $\|\varphi_r\|_{C^2(B_1)}\leq Cr^\epsilon$.

Thus, the function
\[w=v-\varphi_r\]
solves $F(D^2w+D^2\varphi_r)=0$ in $B_2\cap \{x_n>0\}$, and $\min(-F(D^2w),\, w)=0$ on $B_2\cap\{x_n=0\}$.
Therefore, any derivative $\partial_\tau w$, with $\tau\in S^{n-1}\cap \{x_n=0\}$, satisfies
\[M^+(\partial_\tau w)\geq-Cr^\epsilon\quad\textrm{and}\quad M^-(\partial_\tau w)\leq Cr^\epsilon\quad \textrm{in}\quad B_2\setminus\Omega^*,\]
where $\Omega^*:=\{w=0\}\cap \{x_n=0\}\cap B_2$.
Moreover, we have
\[\partial_\tau w=0\quad \textrm{on}\quad \Omega^*.\]

Now, notice that by Proposition \ref{prop-rescalings}, for any given $\delta>0$ we may choose $r>0$ small enough so that $|\partial_\tau w-\partial_\tau u_0|\leq \delta$, where $u_0$ is a global solution of \eqref{eqnliouville}-\eqref{grad}-\eqref{growth}.
By Lemma \ref{u0}, we find
\[\partial_\tau w\geq c_0(\tau\cdot e)-\delta\qquad \textrm{in}\quad\big( \{x'\cdot e\geq\eta\}\cup \{|x_n|\geq\eta\}\big) \cap B_2.\]
Now, choosing $\delta$ small enough (depending on $\ell$), this gives
\[\partial_\tau w\geq \tilde c_0\qquad \textrm{in}\quad \big(\{x'\cdot e\geq\eta\}\cup \{|x_n|\geq\eta\}\big) \cap B_2,\]
for all $\tau\in S^{n-1}\cap\{x_n=0\}$ such that $\tau\cdot e\geq \ell/\sqrt{1+\ell^2}$.
Finally, using Lemma \ref{lem-max} (applied to $\partial_\tau w$) we obtain
\[\partial_\tau w\geq0\quad\textrm{in}\quad B_1,\]
as desired.
\end{proof}

\section{The regular set is open and $C^1$}
\label{sec7}

In this Section, we finally prove Theorem \ref{thm1}.
By Proposition \ref{prop-Lip}, we know that if $x_0$ is a regular point, then the free boundary is $C^1$ at $x_0$.
We next prove that the regular set is open, and this will yield Theorem \ref{thm1}.

In this section $\epsilon_0$ denotes a fixed constant in $(0,1-\beta_2)$, where $\beta_2$ is ``subsolution''  exponent given by Proposition \ref{prop-1D}. For instance, we may fix $\epsilon_0 = \frac12(1-\beta_2)$, a constant depending only on the ellipticity constants (and thus universal).

\begin{prop}\label{contagi}
Assume $0$ is a regular free boundary point with exponent $\epsilon$ and modulus $\nu$.
Then, there is $e\in S^{n-1}\cap \{x_n=0\}$ and there is $r>0$ such that for any free boundary point $x_0\in \partial\{u=\varphi\}\cap \{x_n=0\}\cap B_r$ we have
\[(u-\varphi)(x_0+t e)\geq ct^{2-\epsilon_0}\qquad \textrm{for all}\quad t\in (0,r/2).\]
The constant $c>0$ depends only on $n$, $\epsilon$, $\nu$, and ellipticity constants.
In particular, every free boundary point in $B_r$ is regular, with a uniform exponent $\epsilon=\epsilon_0/2$ and a uniform modulus $\tilde \nu = \tilde\nu(t) = c t^{\epsilon-\epsilon_0}$.
\end{prop}

To prove Proposition \ref{contagi}, we need the following Lemma. Recall that $x'$  denote points in $\R^{n-1}$ and the extension operators $E^+$ and $E^-$ were defined in Proposition~\ref{prop-extensions}.

\begin{lem}\label{homog-subsol}
Let  $e$ be a unit vector in $\R^{n-1}\times\{0\}$, and $0<\beta_1<\frac12<\beta_2<1$ the exponents in Corollary~\ref{cor-1D}.
Define
\[\psi_{\rm sub}(x') :=e\cdot x'- \eta |x'| \left(1- \frac{(e\cdot x')^2}{|x'|^2} \right)\]
\[\psi_{\rm super}(x') :=  e\cdot x' + \eta |x'| \left(1- \frac{(e\cdot x')^2}{|x'|^2} \right),\]
\[\Phi_{\rm sub} := E^-\left[  \left( \psi_{\rm sub }\right)_+^{\beta_2+\gamma} \right] \quad\mbox{and}\quad \Phi_{\rm super} := E^+\left[ \left(\psi_{\rm super }\right)_+^{\beta_1-\gamma}\right].\]

For every $\gamma\in (0, \min\{|\beta_1-0|,|\beta_2-1| \})$ there is $\eta>0$ such that two functions $\Phi_{\rm sub}$ and $\Phi_{\rm super}$ satisfy
\[\begin{cases}
M^- \Phi_{\rm sub} =  0 \quad & \mbox{in }\{x_n>0\}  \\
\partial_{x_n} \Phi_{\rm sub}\ge  c_\gamma d^{\beta_2+\gamma-1}>0 \quad & \mbox{on }\{x_n=0\}\cap\mathcal C_\eta^*  \\
\Phi_{\rm sub} = 0  \quad & \mbox{on }  \{x_n=0\}\setminus \mathcal C_\eta^*
\end{cases}\]
and
\[\begin{cases}
M^+ \Phi_{\rm super} = 0 \quad & \mbox{in }\{x_n>0\}\\
\partial_{x_n} \Phi_{\rm super} \le -c_\gamma d^{\beta_2+\gamma-1}<0 \quad & \mbox{on }\{x_n=0\}\cap \mathcal C_{-\eta}^*  \\
\Phi_{\rm super} = 0  \quad & \mbox{on }\{x_n=0\} \setminus \mathcal C_{-\eta}^*
\end{cases}\]
where $\mathcal C_{\pm\eta}^*$ is the cone
\begin{equation}\label{cone}
\mathcal C_{\pm\eta}^*: = \left\{ (x',0) \in \R^n\ :  e\cdot \frac{x'}{|x'|}    >  \pm \eta \left( 1 -  \left( e\cdot \frac{ x' }{|x'|} \right)^2\right) \right\},
\end{equation}
and $d$ is the distance to $\mathcal C_{\pm\eta}^*$.
The constants $c_\gamma$ and $\eta$ depend only on $\gamma$, $s$, ellipticity constants, and dimension.
\end{lem}

\begin{proof}[Proof of Lemma \ref{homog-subsol}]
We prove the statement for $\Phi_{\rm sub}$.
The statement for $\Phi_{\rm super}$ is proved similarly.

Let us denote $\psi =\psi_{\rm sub}$  and $\Phi = \Phi_{\rm sub}$. Note that $\Phi$ is the $E^-$ extension of a homogeneous function of degree $\beta_2+\gamma$ and thus by uniqueness of the extension (among functions with subcritical growth) it will be homogeneous with the same exponent.

By definition we have $M^- \Phi = 0$ in $\{x_n>0\}$ and  $\Phi = 0$  on $\{x_n=0\} \setminus \mathcal C_{\eta}^*$ since $\psi<0$ on that set.

We thus only need to check that, for $\eta>0$ small enough
\[
\partial_{x_n} \Phi\ge  0 \quad  \mbox{on }\{x_n=0\}
\]

By homogeneity, it  is enough to prove that $\partial_{x_n} \Phi \ge 0$ on points belonging to $e + \partial \mathcal{  C}_\eta^*$, since all the positive dilations of this set with respect to the origin cover the interior of $\mathcal{ C}_\eta^*$.

Let thus $P\in  \partial \mathcal{ C}_\eta^*$, that is,
\[ e\cdot P= \eta \left( |P| -  \frac{(e\cdot P)^2}{|P|} \right) .\]

We note that ---recall that both $P,e\in \{x_n=0\}$
\[\begin{split}
 \psi(P+e+x') &=
 e\cdot (P+e+x')- \eta \left( |P+e+x'| -  \frac{(e\cdot (P+e+x'))^2}{|P+e+x'|} \right)
\\
&=  1 +e\cdot x- \eta \left( |P+e+x| -|P|-  \frac{(e\cdot (P+e+x))^2}{|P+e+x|} +\frac{(e\cdot P)^2}{|P|} \right)\\
&=   1 +e\cdot x'- \eta \psi_P(x')
\end{split}
\]
\[\psi_P(x')  :=  |P+e+x'| -|P|-  \frac{(e\cdot (P+e+x'))^2}{|P+e+x'|} +\frac{(e\cdot P)^2}{|P|} .\]

Then we define
\[\begin{split}
\Phi_{P,\eta}(x) &:= \Phi(P+e+x)
 \\
&=E^-\left[  (x',0)\mapsto  \bigl( 1 +e\cdot x'- \eta \psi_P(x')  \bigr)_+^{\beta_2+\gamma}\right](x),
\end{split}\]
where

Note that the functions $\psi_P$  satisfy
\[\psi_P(0) = 0,\]
\[ |\nabla \psi_P(x')|  \le  C  \quad \mbox{in } \R^n \setminus \{ -P-e\},  \]
and
\[ |D^2 \psi_P(x')|  \le  C  \quad \mbox{for  }x' \in B_{1/2}^*, \]
where $C$ does not depend on $P$ (recall that $|e|=1$).

Then, the (traces of) the family $\Phi_{P,\eta}$ satisfy
\[ \Phi_{P,\eta} \rightarrow  (1+e\cdot x')_+^{\beta_2+\gamma} \quad \mbox{in }C^2(\overline{B_{1/2}^*})\]
as $\eta\searrow 0$, uniformly in $P$.

Moreover,
\[ \bigl|\Phi_{P,\eta}- (1+e\cdot x')_+^{\beta_2+\gamma}  \bigr|  \le  (C \eta |x'|) ^{\beta_2+\gamma} \]
with $C$ independent of $P$.

Thus, since $\beta_2+\gamma<1$,  Lemma \ref{aux} implies
\[ \partial_{x_n} \Phi_{P,\eta}(0)  \rightarrow \partial_{x_n} E^-\left[(x',0)\mapsto(1+e\cdot x' )_+^{s+\gamma}\right](0) = c(s,\gamma,\lambda) >0 ,\]
uniformly in $P$ as $\eta\searrow 0$.

In particular one can chose $\eta = \eta(\gamma,\lambda,\Lambda)$ so that $\partial_{x_n} \Phi_{P,\eta}(0) \ge c(s,\gamma,\lambda)>0$ for all $P\in \partial \mathcal { C}_\eta^*$ and the lemma is proved.
\end{proof}

We can now show Proposition \ref{contagi}.

\begin{proof}[Proof of Proposition \ref{contagi}]
We want to show that  there is $e\in S^{n-1}\cap \{x_n=0\}$ and there is $r>0$ such that for any free boundary point $x_0\in \partial\{u=\varphi\}\cap \{x_n=0\}\cap B_r$ we have
\begin{equation}\label{ourgoal}
(u-\varphi)(x_0+t e)\geq ct^{2-\epsilon_0}\qquad \textrm{for all}\quad t\in (0,r/2).
\end{equation}
This will follow using the subsolitions of Proposition \ref{homog-subsol} and Lemma \ref{lemsubsol}, from a inspection of the Proof of Proposition \eqref{prop-Lip}.
Recall that in all the paper $\epsilon_0$ denotes some constant in $(0,1-\beta_2)$.

Indeed, given $\eta>0$ by Proposition \eqref{prop-Lip} we find $r>0$ such that, for every $x_0 \in \partial\{u=\varphi\}\cap \{x_n=0\}\cap B_{r}$
\begin{equation}\label{conseqProp}
u>\varphi \quad \mbox{on}\quad  B^*_{2r} \cap(x_1+\mathcal C_\eta).
\end{equation}

Then, similarly as in the proof of Proposition \eqref{prop-Lip} the function
\[w(x)=\frac{u(rx)-\varphi(rx)}{\|u\|_{L^\infty(B_r)}},\]
with $r>0$ small satisfies
\[M^+(\partial_e w)\geq-\delta \quad\textrm{and}\quad M^-(\partial_e w)\leq \delta\quad \textrm{in}\quad B_2\setminus \{x_n=0,\,w=0\},\]
where $\delta$ can be arbitrarily small provided that $r$ is small enough.

Moreover, still as in the proof of Proposition \eqref{prop-Lip},  we have
\begin{equation}\label{mar}
\partial_e w\geq c_0>0\qquad \textrm{onn}\quad B_1^*\cap \{x'\cdot e\geq 1/10\}.
\end{equation}

Rescaling \eqref{conseqProp} we that hat, for every $x_0 \in \partial\{w=0\}\cap  B_{1}^*$
\[
\{x_n=0,\,w=0\} \cap B^*_2 \subset B^*_2 \setminus (x_1+\mathcal C_\eta)
\]

Let us fix $\rho=1/10$, $B^* = B^*_{1/4}(e/2)$, and $\gamma\in (\beta_2,1)$ satisgying   $\beta_2+\gamma=1-\epsilon_0$.
Let $\eta\in C^2(\overline {B_{1}})$ be some smooth ``cutoff'' function with $\eta=1$ for $|x|\ge 1-\rho$ and $\eta=0$ in $B_{1/2}$. Let us call
\[C_1:=  \sup_{B_{1}} M^+ \eta = \sup_{B_{1-\rho}} M^+ \eta >0 \]

Let $\phi$ be the subsolution of Lemma \ref{lemsubsol} with $\rho = 1/10$ and
$B^* = B^*_{1/4}(e/2)$.
Let $\Phi=\Phi_{\rm sub}/ \|\Phi_{\rm sub}\|_{L^\infty(B_1)}$ the subsolution of Lemma \ref{homog-subsol} that vanishes in $\R^{n-1}\setminus \mathcal C_{\eta}^*$ and has homogeneity $\beta_2+\gamma$.

Let us fix  $x_0 \in \partial\{w=0\}\cap B_1^*$.

We will show next that, for $C$ large enough,
\begin{equation}\label{done}
  C\partial_e w - (x_n)^2 + 2\eta \ge 2C_1 \phi + \Phi(\,\cdot\,-x_0) \quad \mbox{in }B_1.
\end{equation}

Let
\[v =  \partial_e w  - (x_n)^2 + 2\eta - 2C_1 \phi - \Phi (\,\cdot\,-x_1) .\]

On on hand, let us show that $v\ge 0$ on  $\partial B_1$. Indeed,  we have ($r$ is large) we have $\partial_e w\ge 0$ in $B_1$. Also, $\eta= 1$ for for $|x|\ge 1-\rho$  and thus $\eta-|x|^2 =0$ on $\partial B_1$. Moreover,  recall that  $\phi=0$ on $\partial B_1$ and, since $0\le\Phi\le 1$ in $B_1$, $\eta-\Phi\ge 0$ on $\partial B_1$.

On the other hand, let us show that
\[M^-v \le 0 \quad \mbox{in }(B_1\setminus  B^*)\cup (x_0+\mathcal C_\eta^*) .\]
Indeed,  we have
\[
\begin{split}
M^- v &= M^- ( C \partial_e w  - (x_n)_+^2 + 2\eta - 2C_1 \phi - \Phi )
\\
& \le C M^- (\partial_e w)  -2 \lambda + 2 \sup_{B_{1-\rho}} M^+ \eta - 2C_1 M^- \phi + M^+\Phi (\,\cdot\,-x_0)
\\
&\le C\delta -2 \lambda  + 2C_1\chi _{B_{1-\rho}}   - 2C_1\chi _{B_{1-\rho}} + M^+ \Phi(\,\cdot\,-x_0)
\\
&\le C\delta -2 \lambda
\\
& \le 0
\end{split}
 \]
in $(B_1\setminus  B^*)\cup (x_1+\mathcal C_\eta^*) $  provided that  $C\delta-2n\lambda\le 0$.

That $v\ge 0$ in $B_1^*\setminus  (x_0+\mathcal C_\eta^*)$ is a now a consequence of \eqref{conseqProp} which implies that $w = (x_n)^2= \phi =\Phi  =0$ on that set.
Last,  recalling \eqref{mar} we see that $v\ge 0$ in $B^*$ can be guaranteed by choosing $C$ large (depending only on $c_0$ and universal constants).

Thus, choosing first $C$ large and then $\delta$ small enough so that $C\delta-2n\lambda\le 0$, and using the maximum principle, we prove $v\ge 0$ in $B_1$ and thus that
\[  C \partial_e w  \ge  \Phi (\,\cdot\,-x_0) = \left( \psi_{\rm sub }(\,\cdot\,-x_1) \right)_+^{\beta_2+\gamma} \quad on B^*_{1/2},\]
where $\psi_{\rm sub }$ was defined in Lemma \ref{homog-subsol}.

After rescaling and noting that $\psi_{\rm sub }(te)=t$, this implies that
\[ \partial_e w(te)  \ge ct^{\beta_2+\gamma}=ct^{1-\epsilon_0}>0 \quad \mbox {for }t\in(0,r/2).\]
Thus, \eqref{ourgoal} follows integrating with respect to $t$ (note that $w(0) = \partial_e (0) =0$) .
\end{proof}

Finally, as a consequence of the previous results, we give the:

\begin{proof}[Proof of Theorem \ref{thm1}]
By Proposition \ref{contagi}, the set of regular points is open, and (i) holds at all such points.
Moreover, still by Proposition \ref{contagi}, given any free boundary point $x_0$, there is a ball $B_r(x_0)$ in which all free boundary points are regular, with a common modulus of continuity $\nu$.
Thus, by Proposition \ref{prop-Lip}, the free boundary is $C^1$ at each of these points, with a uniform modulus of continuity (that depends on~$x_0$).
Thus, the free boundary is locally a $C^1$ graph in $B_r(x_0)$.
\end{proof}

When the ellipticity constants $\lambda$ and $\Lambda$ are close to 1, we establish the following.

\begin{cor}\label{cor-reg}
Let $F$ be as in \eqref{F}, and $u$ be any solution of \eqref{obst-pb}, with $\varphi\in C^{1,1}$.
Then, for any small $\delta>0$ we have
\[u\in C^{1,\frac12-\delta}(B_{1/2})\quad \textrm{whenever}\quad |\Lambda-1|+|\lambda-1|\leq \delta/C_0.\]
The constant $C_0$ is universal.
Furthermore, under such assumption on the ellipticity constants, we have
\[\|u\|_{C^{1,\frac12-\delta}(B_{1/2})}\leq C\bigl(\|u\|_{L^\infty(B_1)}+\|\varphi\|_{C^{1,1}(B_1\cap \{x_n=0\})}\bigr),\]
with $C$ depending only on $n$, $\lambda$ and $\Lambda$.
\end{cor}

\begin{proof}
The proof is by contradiction, using the result in Corollary \ref{corclassif}.

Dividing by a constant if necessary, we assume $\|u\|_{L^\infty(B_1)}+\|\varphi\|_{C^{1,1}(B_1\cap \{x_n=0\})}\leq 1$.
We first claim that, for every free boundary point $x_0\in B_{1/2}\cap \partial\{u=\varphi\}$, we have
\begin{equation}\label{claim-final}
\bigl|u(x)-u(x_0)-\nabla u(x_0)\cdot(x-x_0)\bigr|\leq C|x-x_0|^{\frac32-\delta},
\end{equation}
with $C$ depending only on $n$ and $\lambda,\Lambda$.

Let us prove \eqref{claim-final} by contradiction.
Indeed, assume there are sequences of operators $F_k$ as in \eqref{F}, obstacles $\varphi_k$ satisfying $\|\varphi_k\|_{C^{1,1}(B_1\cap \{x_n=0\})}\leq 1$, solutions $u_k$ to \eqref{obst-pb} with $\|u_k\|_{L^\infty(B_1)}\leq1$, and free boundary points $x_k\in B_{1/2}$, such that
\[\bigl|u_k(x)-\nabla u_k(x_k)\bigr|\geq k|x-x_k|^{\frac32-\delta},\]
for all $k\geq1$.
By the $C^{1,\alpha}$ estimates in \cite{Fer}, we know that $|\nabla u_k(x_0)|\leq C$, so that after subtracting a linear function we may assume $u_k(x_k)=0$ and $\nabla u_k(x_k)=0$.
Moreover, after a translation we may assume for simplicity that $x_k=0$.

Then, defining
\[\theta(\rho)=\sup_{\rho\leq r\leq 1}\sup_k r^{\delta-\frac32}\|u_k\|_{L^\infty(B_r)},\]
and by the exact same argument in Lemma \ref{simple-lem}, we find a sequence $r_k\to0$ for which
\[w_k(x)=\frac{u_k(r_kx)}{\|u_k\|_{L^\infty(B_{r_k})}}\]
satisfies
\[|w_k(x)|\leq C\bigl(1+|x|^{\frac32-\delta}\bigr)\quad \textrm{in}\quad B_{1/r_k},\]
$\|w_k\|_{L^\infty(B_1)}=1$, $w_k(0)=0$, $\nabla w_k(0)=0$, and
\[\begin{array}{rclll}
F_k(D^2w_k)&=&0\,&\textrm{in} &B_{1/r_k}\setminus \{x_n=0\}\\
\min(-F_k(D^2w_k),\,w_k-\varphi_k)&=& 0\,&\textrm{on} &B_{1/r_k}\cap \{x_n=0\},
\end{array}\]
where
\[\|\varphi_k\|_{C^{1,1}(B_R)}=\frac{\|\varphi\|_{B_{Rr_k}}}{\|u_k\|_{L^\infty(B_{r_k})}}\leq \frac{CR^2(r_k)^2}{(r_k)^{\frac32-\delta}}=CR^2(r_k)^{\frac12+\delta}\]
converges to 0 for every fixed $R$ as $k\to\infty$.

Thus, by $C^{1,\alpha}$ estimates, up to a subsequence the operators $F_k$ converge to an operator $F$ as in \eqref{F}, and the functions $w_k$ converge locally uniformly to a function $w$ satisfying
\[|w(x)|\leq C\bigl(1+|x|^{\frac32-\delta}\bigr)\quad \textrm{in}\quad \R^n,\]
$\|w\|_{L^\infty(B_1)}=1$, $w(0)=0$, $\nabla w(0)=0$, and
\[\begin{array}{rclll}
F(D^2w)&=&0\,&\textrm{in} &\R^n\setminus \{x_n=0\}\\
\min(-F(D^2w),\,w)&=& 0\,&\textrm{on} &\R^n\cap \{x_n=0\}.
\end{array}\]
By Corollary \ref{corclassif}, we get $w\equiv0$, a contradiction.
Thus, \eqref{claim-final} is proved.

Finally, combining \eqref{claim-final} with interior regularity estimates, the result follows exactly as in the proof of \cite[Theorem~1.1]{Fer}.
\end{proof}

\end{document}